\documentclass[review,hidelinks,onefignum,onetabnum]{siamart250106}
\usepackage{lipsum}
\usepackage{amsfonts}
\usepackage{graphicx}
\usepackage{epstopdf}
\usepackage{algorithmic}
\usepackage{amssymb}
\usepackage{amsmath}
\usepackage{subfig}
\usepackage{amsopn}

\ifpdf
  \DeclareGraphicsExtensions{.eps,.pdf,.png,.jpg}
\else
  \DeclareGraphicsExtensions{.eps}
\fi

\newcommand{\height}{4.5cm}

\def\dint{\mathop{\displaystyle \int}}
\def\dsum{\mathop{\displaystyle \sum }}

\newsiamremark{remark}{Remark}
\newsiamremark{hypothesis}{Hypothesis}
\crefname{hypothesis}{Hypothesis}{Hypotheses}
\newsiamthm{claim}{Claim}
\headers{CCMM for a CIP of the Epidemiology}{M.~V.~Klibanov and T.~Truong}

\ifpdf
\hypersetup{
  pdftitle={Monitoring of Epidemics via Solution of a CIP},
  pdfauthor={M. V. Klibanov and T. Truong}
}
\fi

\begin{document}

\title{The Carleman Contraction Mapping Method for a Coefficient Inverse
Problem of the Epidemiology\thanks{%
Submitted to the editors DATE. \funding{}}}
\author{Michael~V.~Klibanov\thanks{%
Department of Mathematics and Statistics, University of North Carolina at
Charlotte, Charlotte, NC } \and Trung~Truong\thanks{%
Department of Mathematics and Physics, Marshall University, Huntington, WV}
}
\maketitle

\begin{abstract}
It is proposed to monitor spatial and temporal spreads of epidemics via
solution of a Coefficient Inverse Problem for a system of three coupled
nonlinear parabolic equations. To solve this problem numerically, a version
of the so-called Carleman contraction mapping method is developed for this
problem. On each iteration, a linear problem with the incomplete lateral
Cauchy data is solved by the weighted Quasi-Reversibility Method, where the
weight is the Carleman\ Weight Function. This is the function, which is
involved as the weight in the Carleman estimate for the corresponding
parabolic operator. Convergence analysis ensures the global convergence of
this procedure. Numerical results demonstrate an accurate performance of
this technique for noisy data.
\end{abstract}

% REQUIRED

% REQUIRED

% REQUIRED

\section{Introduction}

\label{sec:1}

In 1927 Kermack and McKendrick have proposed a renowned mathematical model
of the spread of epidemics . Their model relies on a system of three coupled
nonlinear Ordinary Differential Equations. This system governs time
dependent propagations of the total number of: susceptible (S), infected (I)
and recovered (R) populations, i.e. it works with the so-called
\textquotedblleft SIR model". In 2021, Lee, Liu, Tembine, Li and Osher \cite%
{Lee} have modified the SIR model of \cite{Kermack} via introducing a system
of three nonlinear coupled parabolic Partial Differential Equations (PDEs).
We also call this model \textquotedblleft SIR system". An attractive point
of the model of \cite{Lee} is that it governs both spatial and time
dependencies of S,I,R populations. Hence, if the solution of this system is
found, then the spatial and temporal distributions of SIR populations are
monitored inside of an affected city.

However, a significant obstacle on this way is that coefficients of that
system are actually unknown. The current paper is devoted to an alleviation
of this obstacle. We solve the corresponding\ Coefficient Inverse Problem
(CIP) of the recovery of those coefficients. This solution can be considered
as the calibration step, although we simultaneously find both: unknown
coefficients and the solution of the SIR system.

Denote $\mathbf{x}=\left( x,y\right) \in \mathbb{R}^{2}$. Let $t>0$ be time.
The unknown coefficients of the SIR system of PDEs of \cite{Lee} are: the
infection rate $\beta \left( \mathbf{x},t\right) ,$ the recovery rate $%
\gamma \left( \mathbf{x},t\right) $ and three 2D vector functions
representing velocities of propagations of S,I,R populations. These
functions are: $q_{S}\left( \mathbf{x},t\right) ,q_{I}\left( \mathbf{x}%
,t\right) $ and $q_{R}\left( \mathbf{x},t\right) .$ In principle, all these
functions need to be computed as solutions of some CIPs. However, this task
is too difficult to handle at this point of time. Hence, we focus below on
computing only infection and recovery rates. If having just generic
dependencies of $\beta \left( \mathbf{x},t\right) $ and $\gamma \left( 
\mathbf{x},t\right) $ on three variables $\left( x,y,t\right) ,$ then we
would need to have input data to be dependent on three variables: otherwise
our input data, which are the boundary functions depending on two variables,
would be under-determined. The latter would most likely imply non-uniqueness
of our CIP. We point out that we work with the case of the data resulting
from a single measurement event, which is quite natural here.

Therefore, we assume below that $\beta =\beta \left( \mathbf{x}\right) $ and 
$\gamma =\gamma \left( \mathbf{x}\right) .$ In the reality of an epidemic,
this assumption might work well only on a small time interval. Then,
however, one can consider the case when $\beta \left( \mathbf{x},t\right) $
and $\gamma \left( \mathbf{x},t\right) $ are piecewise constant functions
with respect to $t$ on a sequence of time intervals:%
\begin{equation}
\left. 
\begin{array}{c}
0<t_{1}<t_{2}<...<t_{n}=T, \\ 
\beta \left( \mathbf{x},t\right) =\beta _{i}\left( \mathbf{x}\right) ,\text{ 
}\gamma \left( \mathbf{x},t\right) =\gamma _{i}\left( \mathbf{x}\right) ,%
\text{ }i=1,...,n-1.%
\end{array}%
\right.  \label{1.1}
\end{equation}%
Our method uses the knowledge of incomplete lateral Cauchy data for the SIR
system. \textquotedblleft Incomplete" in this case means that the Neumann
boundary condition is given on the entire lateral boundary of our domain of
interest, and the Dirichlet boundary condition is given only on a part of
that boundary. However, our numerical method does not need a knowledge of
the initial conditions of that system. Instead, it needs the knowledge of
the solution of the SIR system at a certain interior point of the time
interval we work with. An applied justification of the latter is given below
in this section. Therefore, if the solution of that system is known at a
sequence of interior points $\left\{ t=t_{0,i}\right\}
_{i=1}^{n-1},t_{0,i}\in \left( t_{i},t_{i+1}\right) ,$ then the method of
the current paper can be adjusted to this case. However, the case (\ref{1.1}%
) is outside of the scope of our publication. In addition, we assume below
that the velocities depend only on $\mathbf{x}$. This assumption can also be
adjusted to the $\left( \mathbf{x},t\right) -$dependence, similarly with the
above. Summarizing, we assume that 
\begin{equation}
\left. 
\begin{array}{c}
\beta =\beta \left( \mathbf{x}\right) ,\gamma =\gamma \left( \mathbf{x}%
\right) , \\ 
q_{S}=q_{S}\left( \mathbf{x}\right) ,q_{I}=q_{I}\left( \mathbf{x}\right)
,q_{R}=q_{R}\left( \mathbf{x}\right) .%
\end{array}%
\right.  \label{1.2}
\end{equation}

Any CIP for a PDE is both ill-posed and nonlinear. The ill-posedness and the
nonlinearity of CIPs cause the non-convexity of conventional least squares
cost functionals for CIPs, see, e.g. \cite{B1,B2,B3,Chavent,Gonch1,Gonch2}.
The non convexity, in turn causes the presence of local minima and ravines
in those functionals. Hence, convergence of any optimization technique for
such a functional can be rigorously guaranteed only if its starting point is
located in a sufficiently small neighborhood of the true solution, which
means local convergence. In general, it is unclear, however, how to figure
out such a neighborhood a priori.

We call a numerical method for a CIP \textquotedblleft globally convergent"
if its convergence to the true solution of that CIP is rigorously guaranteed
without an assumption of an advanced knowledge of a sufficiently small
neighborhood of that solution.

To address the problem of the construction of globally convergent numerical
methods for CIPs, the convexification concept was first proposed in \cite{KI}%
. The convexification is a numerical concept for CIPs, which is based on
Carleman estimates. Naturally, each new CIP requires its own version of the
convexification method, see, e.g. publications \cite{Bak,KL,Epid,SAR,KTR}
about various versions of the convexification method. These versions include
both the global convergence analysis and numerical studies. The global
convergence property is defined above. In particular, in \cite{Epid} the CIP
of this paper is solved numerically by the convexification method.

Although we study here the same CIP\ as the one in \cite{Epid}, there is a
significant difference between the approach of this paper and the one of 
\cite{Epid}. This is because we use here the so-called \textquotedblleft
Carleman contraction mapping method" (CCMM). The CCMM was originated in \cite%
{LN} and \cite{Nguyen}, also, see, e.g. \cite{Le,KN} for some follow up
publications. This method is applicable to some inverse source problems for
quasilinear PDEs, see, e.g. \cite{KN,LN,Nguyen}. It is also applicable to
some CIPs, see, e.g. \cite{Le}. In the CCMM,\ a certain nonlinear boundary
value problem associated with the original problem (see, section 3 for our
specific case) is solved numerically via an iterative procedure. On each
iterative step a weighted quadratic functional is minimized. It seems to be,
on the first glance, that such an approach should diverge. However, the key
element of the CCMM is the presence of the so-called Carleman Weight
Function (CWF) in that functional. The CWF is the function, which is
involved in the Carleman estimate for the corresponding PDE operator. The
presence of the CWF guarantees convergence of the process. That functional
is minimized via a version of the so-called Quasi-Reversibility Method
(QRM). The QRM\ was first proposed by Lattes and Lions \cite{LL}. The
version of the QRM, which requires the minimization of a quadratic
functional, was first proposed in \cite[section 2.5]{KT}, also see \cite[%
chapter 4]{KL}.

The convergence rate of the CCMM is the same as the one of the classical
contraction mapping, see Theorem 7.1 in section \ref{sec:7} for the main
convergence result of this paper. The latter explains the name of this
method. The idea of the proof of Theorem 7.1 is similar with the ideas of
proofs in \cite{LN,Nguyen}. Nevertheless, significant differences with \cite%
{LN,Nguyen} remain in this paper. They are caused by significant differences
between the underlying PDE operators. Therefore, it is necessary to provide
a complete proof of that theorem here. The CCMM converges globally in terms
of the above definition. Just as in the case of the convexification method,
each new application of the CCMM requires its own global convergence
analysis. The latter is the focal point of the analytical effort of this
paper.

\begin{remark} 
  \label{Remarks 1.1}
  \begin{enumerate}
  \item To simplify the presentation of solving our already difficult
  CIP, we assume that all functions representing the solution of our forward
  problem are sufficiently smooth.

  \item In addition to item 1, it is well known that the regularity
  assumptions are not of a significant concern in the theory of CIPs, see,
  e.g. \cite{Nov1,Nov2}, \cite[Theorem 4.1]{Rom}.
  \end{enumerate}
\end{remark}

All functions considered below are real valued ones. In section 2 we pose
both forward and inverse problems. In section 3 we describe our nonlinear
transformation procedure. In section 4 we formulate two estimates: a
Carleman estimate and an estimate of the Volterra integral, in which our CWF
is involved. In section 5 we formulate our version of the second generation
of the convexification method. In sections 6 and 7 we carry out the
convergence analysis. In particular, Theorems 7.1-7.4 of section 7 are the
main analytical results of this paper. In section 8 we present results of
our numerical experiments.

\section{Statements of Forward and Inverse Problems}

\label{sec:2}

In this section we pose both the forward problem and the inverse problem.
Let $m>1$ be an integer and $B$ be a Banach space with its norm $\left\Vert
\cdot \right\Vert _{B}$. Denote%
\begin{equation}
\left. 
\begin{array}{c}
B_{m}=B\times B\times ...\times B,\ m\mbox{ times,} \\ 
\left\Vert u\right\Vert _{B_{m}}=\left( \sum\limits_{k=1}^{m}\left\Vert
u_{k}\right\Vert _{B}^{2}\right) ^{1/2},\ \forall u=\left(
u_{1},...,u_{m}\right) ^{T}\in B_{m}.%
\end{array}%
\right.  \label{2.0}
\end{equation}

There is a general Carleman estimate for parabolic equations, in which the
CWF depends on two large parameters \cite[section 2.3]{KL}, \cite[\S 1 of
Chapter 4]{LRS}. This estimate is applicable to those CIPs for parabolic
PDEs, which have lateral Cauchy data at any small smooth part of the spatial
boundary. \ However, since the CWF in this case depends on two large
parameters, then it changes too rapidly. This, in turn makes it inconvenient
to work with such a CWF in numerical studies. Thus, we use a simpler CWF of 
\cite[formula (9.20)]{KL}, which depends on only one large parameter.
However, to use it, we need to assume that we solve our CIP in a rectangular
domain.

Let the number $\alpha $ be such that 
\begin{equation}
\alpha \in \left( 0,1/\sqrt{2}\right) .  \label{1}
\end{equation}
Let $a,b,A>0$ be some numbers. Denote 
\begin{equation}
\left. 
\begin{array}{c}
\Omega =\left\{ \mathbf{x}=\left( x,y\right) :a<x<b,\left\vert y\right\vert
<A\right\} ,\mbox{ } \\ 
\Gamma =\partial \Omega \cap \left\{ x=b\right\} , \\ 
Q_{T}=\Omega \times \left( 0,T\right) ,\mbox{ }S_{T}=\partial \Omega \times
\left( 0,T\right) ,\mbox{ }\Gamma _{T}=\Gamma \times \left( 0,T\right) , \\ 
Q_{\alpha T}=\Omega \times \left( T/2\left( 1-\alpha \right) ,T/2\left(
1+\alpha \right) \right) .%
\end{array}%
\right.  \label{2.1}
\end{equation}%
On the other hand, let $G\subset \mathbb{R}^{2}$ be a bounded domain with a
sufficiently smooth boundary and such that 
\begin{equation}
\Omega \subset G,\text{ }\partial \Omega \cap \partial G=\varnothing .
\label{2.01}
\end{equation}%
Denote%
\begin{equation}
G_{T}=G\times \left( 0,T\right) ,\text{ }S_{G_{T}}=\partial G\times \left(
0,T\right) .  \label{2.02}
\end{equation}

We now use notations of \cite{Lee}. Assume that (\ref{1.2}) holds. Also, let%
\begin{equation}
\beta ,\gamma \in C\left( \overline{G}\right) ;\text{ }q_{S},q_{I},q_{R}\in
C^{1}\left( \overline{G}\right) .  \label{2.03}
\end{equation}%
Let $\rho _{S}\left( \mathbf{x},t\right) ,\rho _{I}\left( \mathbf{x}%
,t\right) $ and $\rho _{R}\left( \mathbf{x},t\right) $ be S,I and R
populations respectively. The initial boundary value problem for the SIR
system of \cite[formulas (2.1)]{Lee} is: 
\begin{equation}
\partial _{t}\rho _{S}-\frac{\eta _{S}^{2}}{2}\Delta \rho _{S}+\mbox{div}%
\rho _{S}q_{S}+\beta \left( \mathbf{x}\right) \rho _{S}\rho _{I}=0,\ \left( 
\mathbf{x},t\right) \in G_{T},  \label{2.2}
\end{equation}%
\begin{equation}
\partial _{t}\rho _{I}-\frac{\eta _{I}^{2}}{2}\Delta \rho _{I}+\mbox{div}%
\left( \rho _{I}q_{I}\right) -\beta \left( \mathbf{x}\right) \rho _{S}\rho
_{I}=0,\ \left( \mathbf{x},t\right) \in G_{T},  \label{2.3}
\end{equation}%
\begin{equation}
\partial _{t}\rho _{R}-\frac{\eta _{R}^{2}}{2}\Delta \rho _{R}+\mbox{div}%
\left( \rho _{R}q_{R}\right) -\gamma \left( \mathbf{x}\right) \rho _{I}=0,\
\left( \mathbf{x},t\right) \in G_{T},  \label{2.4}
\end{equation}%
\begin{equation}
\partial _{n}\rho _{S}\mathbf{\mid }_{S_{G_{T}}}=g_{1}\left( \mathbf{x}%
,t\right) ,\ \partial _{n}\rho _{I}\mathbf{\mid }_{S_{G_{T}}}=g_{2}\left( 
\mathbf{x},t\right) ,\ \partial _{n}\rho _{R}\mathbf{\mid }%
_{S_{G_{T}}}=g_{3}\left( \mathbf{x},t\right) ,  \label{2.5}
\end{equation}%
\begin{equation}
\rho _{S}\left( \mathbf{x},0\right) =\rho _{S}^{0}\left( \mathbf{x}\right)
,\ \rho _{I}\left( \mathbf{x},0\right) =\rho _{I}^{0}\left( \mathbf{x}%
\right) ,\ \rho _{R}\left( \mathbf{x},0\right) =\rho _{R}^{0}\left( \mathbf{x%
}\right) ,\mathbf{x}\in G.  \label{2.6}
\end{equation}%
Here and below $\partial _{n}$ is the normal derivative. This is our forward
problem. Functions $g_{1},g_{2},g_{3}$ are fluxes of S,I,R populations
through the boundary $\partial G$ \cite{Lee}. Here, $\eta _{S}^{2},\eta
_{R}^{2},\eta _{R}^{2}>0$ are constant viscosity terms. To simplify the
presentation, we assume below that%
\begin{equation}
\frac{\eta _{S}^{2}}{2}\equiv \frac{\eta _{I}^{2}}{2}\equiv \frac{\eta
_{R}^{2}}{2}\equiv d>0,  \label{2.7}
\end{equation}%
where $d$ is a number. A more general case when identities (\ref{2.7}) are
not in place, can be considered along the same lines, although this topic is
outside of the scope of the current publication.

If system (\ref{2.2})-(\ref{2.4}) would be replaced with a single linear
parabolic equation, then Theorem 5.3 of \S 5 of Chapter 4 of the book \cite%
{Lad} would guarantee the existence and uniqueness of a sufficiently smooth
solution of the corresponding initial boundary value problem with the
Neumann boundary condition, provided that coefficients, initial and boundary
data satisfy certain non-restrictive conditions. However, since we now have
a system of coupled parabolic equations and these equations are nonlinear
ones, then basically nothing can be derived from the book \cite{Lad} about
existence of a sufficiently smooth solution of the initial boundary value
problem (\ref{2.2})-(\ref{2.6}). An improvement of the results of this book,
which is unlikely possible, is outside of the scope of this paper.

On the other hand, we truly want to solve the applied CIP posed below since
this solution might bring at least something of a value to the important
societal problem of monitoring of epidemics. Therefore, we just assume the
existence and uniqueness of the solution 
\begin{equation}
\left( \rho _{S},\rho _{I},\rho _{R}\right) \left( \mathbf{x},t\right) \in
C_{3}^{8,4}\left( \overline{G}_{T}\right)  \label{2.8}
\end{equation}%
of the forward problem (\ref{2.2})-(\ref{2.6}), see (\ref{2.0}) for the
subscript \textquotedblleft 3" in $C_{3}^{8,4}\left( \overline{G}_{T}\right)
.$ As to the $C_{3}^{8,4}\left( \overline{G}_{T}\right) -$smoothness, we
need it for our derivations and refer to Remarks 1.1 in section 1.

We now pose our CIP. It makes a little sense to measure the S,I,R functions
at $\left\{ t=0\right\} .$ Indeed, the epidemic process is not yet mature at
small times. Hence, it is unlikely that the authorities know about the
existence of an epidemic at $t\approx 0.$ Hence, unlike the forward problem,
we do not assume in our inverse problem the knowledge of initial conditions (%
\ref{2.6}) at $\left\{ t=0\right\} .$ Instead, to work with our CIP, we
assume the knowledge of the functions $\rho _{S},\rho _{I}$,$\rho _{R}$ at a
fixed moment of time $t_{0}\in \left( 0,T\right) .$ For convenience of
notations, we set $t_{0}=T/2.$ Thus, we assume below that the spatial
distributions of S,I,R populations inside of the affected city are conducted
at the moment of time $t=T/2.$ As to the boundary data for our CIP, we need
these data at $S_{T}$ and $\Gamma _{T},$ see (\ref{2.1}). This is because we
actually solve the CIP only inside the domain $\Omega ,$ see the paragraph
below (\ref{2.0}). Thus, it makes sense to assume that the boundary data for
our CIP are generated by the solution of of the forward problem (\ref{2.2})-(%
\ref{2.6}).

\textbf{Coefficient Inverse Problem (CIP).} \emph{Assume that conditions (%
\ref{1.2}), (\ref{2.1})-(\ref{2.03}) hold, coefficients }$\beta \left( 
\mathbf{x}\right) $\emph{\ and }$\gamma \left( \mathbf{x}\right) $\emph{\
are known for }$\mathbf{x}\in G\diagdown \Omega $\emph{\ and are unknown
inside of the domain }$\Omega .$\emph{\ In addition, assume that vector
functions }$q_{S},q_{I},q_{R}$\emph{\ are known. Let the vector function }$%
\left( \rho _{S},\rho _{I},\rho _{R}\right) \left( \mathbf{x},t\right) \in
C_{3}^{8,4}\left( \overline{G}_{T}\right) $\emph{\ (see (\ref{2.8})) be the
solution of the forward problem (\ref{2.2})-(\ref{2.6}). Let }%
\begin{equation}
\rho _{S}\left( \mathbf{x},\frac{T}{2}\right) =p_{1}\left( \mathbf{x}\right)
,\ \rho _{I}\left( \mathbf{x},\frac{T}{2}\right) =p_{2}\left( \mathbf{x}%
\right) ,\ \rho _{R}\left( \mathbf{x},\frac{T}{2}\right) =p_{3}\left( 
\mathbf{x}\right) ,\text{ }\mathbf{x}\in \Omega ,  \label{2.9}
\end{equation}%
\begin{equation}
\left. 
\begin{array}{c}
\partial _{n}\rho _{S}\left( \mathbf{x},t\right) \mid _{S_{T}}=r_{1}\left( 
\mathbf{x},t\right) ,\text{ }\partial _{n}\rho _{I}\left( \mathbf{x}%
,t\right) \mid _{S_{T}}=r_{2}\left( \mathbf{x},t\right) , \\ 
\partial _{n}\rho _{R}\left( \mathbf{x},t\right) \mid _{S_{T}}=r_{3}\left( 
\mathbf{x},t\right) ,%
\end{array}%
\right.  \label{2.10}
\end{equation}%
\begin{equation}
\left. 
\begin{array}{c}
\rho _{S}\mid _{\Gamma _{T}}=f_{1}\left( y,t\right) ,\rho _{I}\mid _{\Gamma
_{T}}=f_{2}\left( y,t\right) , \\ 
\rho _{S}\mid _{\Gamma _{T}}=f_{3}\left( y,t\right) ,\left( y,t\right) \in
\Gamma _{T}.%
\end{array}%
\right.  \label{2.11}
\end{equation}%
\emph{Assume that functions in the right hand sides of (\ref{2.9})-(\ref%
{2.11}) are known. In addition, assume that there exists a number }$c>0$%
\emph{\ such that}{\ }%
\begin{equation}
\left\vert p_{1}\left( \mathbf{x}\right) \right\vert ,\left\vert p_{2}\left( 
\mathbf{x}\right) \right\vert \geq c,\ \mathbf{x}\in \overline{\Omega }.
\label{2.12}
\end{equation}%
\emph{At the same time, a knowledge of initial conditions (\ref{2.11}) is
not assumed. Determine the infection rate }$\beta \left( \mathbf{x}\right) $%
\emph{\ and the recovery rate }$\gamma \left( \mathbf{x}\right) $\emph{\ for}%
{\ }$\mathbf{x}\in \Omega .$

In the case of a single linear parabolic equation, one can sometimes impose
sufficient conditions to guarantee an analog of (\ref{2.12}). This is
usually done using the maximum principle for parabolic equations \cite[\S 2
of Chapter 1]{Lad}. However, since we have the system (\ref{2.2})-(\ref{2.4}%
) of nonlinear equations, then it is unlikely that a proper sufficient
condition guaranteeing (\ref{2.12}) can be imposed. \ Thus, we simply assume
the validity of (\ref{2.12}). We note that condition (\ref{2.12}) has a
clear physical meaning: it tells one that the S and I populations exceed a
certain number at the time when measurements inside of the affected city are
conducted.

\section{Transformation}

\label{sec:3}

The first step of both first and second generations of the convexification
method is a transformation procedure. In our specific case, this procedure
transforms the original CIP into such a system of six coupled nonlinear
integral differential equations with Volterra integral operators in them,
which does not contain unknown coefficients.

Set $t=T/2$ in (\ref{2.2}) and (\ref{2.4}). Using (\ref{2.9}) and (\ref{2.12}%
), we obtain%
\begin{equation}
\beta \left( \mathbf{x}\right) =-\frac{1}{\left( p_{1}p_{2}\right) \left( 
\mathbf{x}\right) }\partial _{t}\rho _{S}\left( \mathbf{x},\frac{T}{2}%
\right) +  \label{3.2}
\end{equation}%
\begin{equation*}
\frac{1}{\left( p_{1}p_{2}\right) \left( \mathbf{x}\right) }\left[ c\Delta
p_{1}\left( \mathbf{x}\right) -\mbox{div}\left( p_{1}q_{S}\right) \left( 
\mathbf{x}\right) \right] ,
\end{equation*}%
\begin{equation}
\gamma \left( \mathbf{x}\right) =\frac{1}{p_{2}\left( \mathbf{x}\right) }%
\partial _{t}\rho _{R}\left( \mathbf{x},\frac{T}{2}\right) -\frac{1}{%
p_{2}\left( \mathbf{x}\right) }\left[ c\Delta p_{3}\left( \mathbf{x}\right) -%
\mbox{div}\left( p_{3}q_{R}\right) \left( \mathbf{x}\right) \right] .
\label{3.3}
\end{equation}%
Denote 
\begin{equation}
\left. 
\begin{array}{c}
v_{1}\left( \mathbf{x,}t\right) =\partial _{t}\rho _{S}\left( \mathbf{x,}%
t\right) ,v_{2}\left( \mathbf{x,}t\right) =\partial _{t}\rho _{I}\left( 
\mathbf{x,}t\right) ,v_{3}\left( \mathbf{x,}t\right) =\partial _{t}\rho
_{R}\left( \mathbf{x,}t\right) , \\ 
V\left( \mathbf{x,}t\right) =\left( v_{1},v_{2},v_{3}\right) ^{T}\left( 
\mathbf{x,}t\right) .%
\end{array}%
\right.  \label{3.4}
\end{equation}%
By (\ref{2.9}) and (\ref{3.4}) 
\begin{equation}
\left. 
\begin{array}{c}
\rho _{S}\left( \mathbf{x,}t\right) =\int\limits_{T/2}^{t}v_{1}\left( 
\mathbf{x,}\tau \right) d\tau +p_{1}\left( \mathbf{x}\right) ,\ \left( 
\mathbf{x,}t\right) \in Q_{T}, \\ 
\rho _{I}\left( \mathbf{x,}t\right) =\int\limits_{T/2}^{t}v_{2}\left( 
\mathbf{x,}\tau \right) d\tau +p_{2}\left( \mathbf{x}\right) ,\ \left( 
\mathbf{x,}t\right) \in Q_{T}, \\ 
\rho _{R}\left( \mathbf{x,}t\right) =\int\limits_{T/2}^{t}v_{3}\left( 
\mathbf{x,}\tau \right) d\tau +p_{3}\left( \mathbf{x}\right) ,\ \left( 
\mathbf{x,}t\right) \in Q_{T}.%
\end{array}%
\right.  \label{3.5}
\end{equation}%
Since by (\ref{3.4})%
\begin{equation*}
\left. 
\begin{array}{c}
\partial _{t}\rho _{S}\left( \mathbf{x,}T/2\right) =v_{1}\left( \mathbf{x,}%
T/2\right) =v_{1}\left( \mathbf{x,}t\right) -\int\limits_{T/2}^{t}\partial
_{t}v_{1}\left( \mathbf{x,}\tau \right) d\tau , \\ 
\partial _{t}\rho _{R}\left( \mathbf{x,}T/2\right) =v_{3}\left( \mathbf{x,}%
T/2\right) =v_{3}\left( \mathbf{x,}t\right) -\int\limits_{T/2}^{t}\partial
_{t}v_{3}\left( \mathbf{x,}\tau \right) d\tau ,%
\end{array}%
\right.
\end{equation*}%
then (\ref{3.2})-(\ref{3.5}) imply:%
\begin{equation}
\left. 
\begin{array}{c}
\beta \left( \mathbf{x}\right) =\left( v_{1}\left( \mathbf{x,}t\right)
-\int\limits_{T/2}^{t}\partial _{t}v_{1}\left( \mathbf{x,}\tau \right) d\tau
\right) s_{1}\left( \mathbf{x}\right) +s_{2}\left( \mathbf{x}\right) , \\ 
\gamma \left( \mathbf{x}\right) =\left( v_{3}\left( \mathbf{x,}t\right)
-\int\limits_{T/2}^{t}\partial _{t}v_{3}\left( \mathbf{x,}\tau \right) d\tau
\right) s_{3}\left( \mathbf{x}\right) +s_{4}\left( \mathbf{x}\right) , \\ 
s_{1}\left( \mathbf{x}\right) =-\left[ \left( p_{1}p_{2}\right) \left( 
\mathbf{x}\right) \right] ^{-1}, \\ 
s_{2}\left( \mathbf{x}\right) =-s_{1}\left( \mathbf{x}\right) \left[ c\Delta
p_{1}\left( \mathbf{x}\right) -\mbox{div}\left( p_{1}q_{S}\right) \left( 
\mathbf{x}\right) \right] , \\ 
s_{3}\left( \mathbf{x}\right) =1/p_{2}\left( \mathbf{x}\right) ,\  \\ 
s_{4}\left( \mathbf{x}\right) =s_{3}\left( \mathbf{x}\right) \left[ c\Delta
p_{3}\left( \mathbf{x}\right) -\mbox{div}\left( p_{3}q_{R}\right) \left( 
\mathbf{x}\right) \right] .%
\end{array}%
\right.  \label{3.6}
\end{equation}%
Differentiate (\ref{2.2})-(\ref{2.4}), (\ref{2.10}) and (\ref{2.11}) with
respect to $t$ and use (\ref{2.7}) and (\ref{3.4})-(\ref{3.6}). We obtain
the system of three coupled nonlinear integral differential equations with
incomplete lateral Cauchy data and Volterra integrals with respect to $t$:%
\begin{equation}
\partial _{t}V-d\Delta V+\left( \mathop{\rm div}\nolimits\left(
v_{1}q_{S}\right) ,\mathop{\rm div}\nolimits\left( v_{2}q_{I}\right) ,%
\mathop{\rm div}\nolimits\left( v_{3}q_{I}\right) \right) ^{T}+P\left(
V\right) =0\text{ in }Q_{T},  \label{3.7}
\end{equation}%
\begin{equation}
\partial _{n}V\mid _{S_{T}}=\left( \partial _{t}r_{1},\partial
_{t}r_{2},\partial _{t}r_{3}\right) ^{T}\left( \mathbf{x},t\right) ,\ V\mid
_{\Gamma _{T}}=\left( \partial _{t}f_{1},\partial _{t}f_{2},\partial
_{t}f_{3}\right) ^{T}\left( y,t\right) .  \label{3.8}
\end{equation}%
In (\ref{3.7}), we have singled out the linear part of the matrix
differential operator. The 3D vector function $P$ is nonlinear with respect
to $V,$%
\begin{equation}
P\left( V\right) =P\left( V,\nabla V,\int\limits_{T/2}^{t}V\left( \mathbf{x}%
,\tau \right) d\tau ,\int\limits_{T/2}^{t}\partial _{t}V\left( \mathbf{x}%
,\tau \right) d\tau ,S\left( \mathbf{x}\right) \right) ,  \label{3.9}
\end{equation}%
\begin{equation}
S=\left( s_{1},s_{2},s_{3},s_{4}\right) ^{T}\left( \mathbf{x}\right) \in
C_{4}^{6}\left( \overline{\Omega }\right) ,  \label{3.90}
\end{equation}%
where functions $s_{i}\left( \mathbf{x}\right) ,$ $i=1,...,4$ are defined in
(\ref{3.6}). The 3D vector function $P$ is 
\begin{equation}
P=P\left( x_{1},...,x_{19}\right) \in C_{3}^{2}\left( \mathbb{R}^{19}\right)
.  \label{3.10}
\end{equation}%
Thus, we have obtained the system of three coupled nonlinear integral
differential equations (\ref{3.7}) with the lateral incomplete Cauchy data (%
\ref{3.8}) and condition (\ref{3.9}). We observe that the unknown
coefficients $\beta \left( \mathbf{x}\right) $ and $\gamma \left( \mathbf{x}%
\right) $ are not involved in system (\ref{3.7}), which is exactly the goal
of our transformation procedure.

However, the presence in (\ref{3.9}) of integral terms containing the $t-$%
derivative $\partial _{t}V\left( \mathbf{x},\tau \right) $ makes the
convergence analysis inconvenient. Therefore, we differentiate (\ref{3.7}), (%
\ref{3.8}) with respect to $t$ and denote 
\begin{equation}
W\left( \mathbf{x},t\right) =\left( V,V_{t}\right) ^{T}\left( \mathbf{x}%
,t\right) =\left( v_{1},v_{2},v_{3},v_{1t},v_{2t},v_{3t}\right) ^{T}\left( 
\mathbf{x},t\right) .  \label{3.11}
\end{equation}%
The system (\ref{3.7}) becomes%
\begin{equation}
\left. 
\begin{array}{c}
\partial _{t}W-d\Delta W+ \\ 
+\left( \mathop{\rm div}\nolimits\left( v_{1}q_{S}\right) ,\mathop{\rm div}%
\nolimits\left( v_{2}q_{I}\right) ,\mathop{\rm div}\nolimits\left(
v_{3}q_{I}\right) ,\mathop{\rm div}\nolimits\left( v_{1t}q_{S}\right) ,%
\mathop{\rm div}\nolimits\left( v_{2t}q_{I}\right) ,\mathop{\rm div}%
\nolimits\left( v_{3t}q_{I}\right) \right) ^{T}+ \\ 
+Y\left( W,S\right) =0\text{ in }Q_{T},%
\end{array}%
\right.  \label{3.12}
\end{equation}%
\begin{equation}
Y\left( W,S\right) =Y\left( W,\int\limits_{T/2}^{t}W\left( \mathbf{x},\tau
\right) d\tau ,\nabla W,\int\limits_{T/2}^{t}\nabla W\left( \mathbf{x},\tau
\right) d\tau ,S\left( \mathbf{x}\right) \right) .  \label{3.13}
\end{equation}%
The boundary conditions for $W$ are generated by (\ref{3.8}) and (\ref{3.11}%
):%
\begin{equation}
\left. 
\begin{array}{c}
W\mid _{\Gamma _{T}}=G_{0}\left( y,t\right) =\left( \partial
_{t}f_{1},\partial _{t}f_{2},\partial _{t}f_{3},\partial
_{t}^{2}f_{1},\partial _{t}^{2}f_{2},\partial _{t}^{2}f_{3}\right)
^{T}\left( y,t\right) , \\ 
\partial _{n}W\mid _{S_{T}}=G_{1}\left( \mathbf{x},t\right) =\left( \partial
_{t}r_{1},\partial _{t}r_{2},\partial _{t}r_{3},\partial
_{t}^{2}r_{1},\partial _{t}^{2}r_{2},\partial _{t}^{2}r_{3}\right)
^{T}\left( \mathbf{x},t\right) .%
\end{array}%
\right.  \label{3.14}
\end{equation}%
By (\ref{3.90}) and (\ref{3.10}) 
\begin{equation}
Y=Y\left( x_{1},...,x_{40}\right) \in C_{6}^{2}\left( \mathbb{R}^{40}\right)
.  \label{3.15}
\end{equation}

The transformation procedure ends up with problem (\ref{3.11})-(\ref{3.15}).
We focus below on the numerical solution of this problem. Suppose for a
moment that problem (\ref{3.11})-(\ref{3.15}) is solved. Then the target
coefficients $\beta \left( \mathbf{x}\right) $ and $\gamma \left( \mathbf{x}%
\right) $ should be reconstructed via backwards calculations using formulas (%
\ref{3.11}), (\ref{3.4}) and (\ref{3.6}) sequentially.

\section{Two Estimates}

\label{sec:4}

As stated in section 2, we use the CWF of \cite[formula (9.20)]{KL}, which
depends only on a single large parameter. Our CWF is%
\begin{equation}
\varphi _{\lambda }\left( \mathbf{x},t\right) =\exp \left[ 2\lambda \left(
x^{2}-\left( t-T/2\right) ^{2}\right) \right] ,  \label{4.1}
\end{equation}%
where $\lambda \geq 1$ is a large parameter. By (\ref{2.1}) and (\ref{4.1})%
\begin{equation}
\left. 
\begin{array}{c}
\varphi _{\lambda }\left( \mathbf{x},t\right) \leq \exp \left[ 2\lambda
\left( b^{2}-\left( t-T/2\right) ^{2}\right) \right] \mbox{
in }Q_{T}, \\ 
\varphi _{\lambda }\left( \mathbf{x},t\right) \geq \exp \left[ -\lambda
\alpha ^{2}T^{2}/2\right] \mbox{ in }Q_{\alpha T}.%
\end{array}%
\right.  \label{4.2}
\end{equation}%
Let $H_{0}^{2,1}\left( Q_{T}\right) $ the subspace of the space $%
H^{2,1}\left( Q_{T}\right) $ defined as:%
\begin{equation}
H_{0}^{2,1}\left( Q_{T}\right) =\left\{ u\in H^{2,1}\left( Q_{T}\right)
:\partial _{n}u\mid _{S_{T}}=0,\ u\mid _{\Gamma _{T}}=0\right\} .
\label{4.3}
\end{equation}

\begin{theorem}[Carleman estimate for the operator $\partial
  _{t}-d\Delta $ - Theorem 9.4.1 in \cite{KL}]\label{Theorem 4.1} {Let }$d>0${\ be the
number in (\ref{2.7}). There exists a sufficiently large number }$\lambda
_{0}=\lambda _{0}\left( Q_{T},d\right) \geq 1${\ and a number }$%
C=C\left( Q_{T},d\right) >0,${\ both depending only on listed
parameters, such that the following Carleman estimate holds:}%
\begin{equation}
  \left. 
  \begin{array}{c}
  \int\limits_{Q_{T}}\left( u_{t}-d\Delta u\right) ^{2}\varphi _{\lambda }d%
  \mathbf{x}dt\geq C\int\limits_{Q_{T}}\left( \lambda \left( \nabla u\right)
  ^{2}+\lambda ^{3}u^{2}\right) \varphi _{\lambda }d\mathbf{x}dt- \\ 
  -C\left( \left\Vert u\left( \mathbf{x},T\right) \right\Vert _{H^{1}\left(
  \Omega \right) }^{2}+\left\Vert u\left( \mathbf{x},0\right) \right\Vert
  _{H^{1}\left( \Omega \right) }^{2}\right) \lambda ^{2}\exp \left( -2\lambda
  \left( T^{2}/4-b^{2}\right) \right) , \\ 
  \forall u\in H_{0}^{2,1}\left( Q_{T}\right) ,\ \forall \lambda \geq \lambda
  _{0}.%
  \end{array}%
  \right.   \label{4.30}
  \end{equation}
\end{theorem}

This Carleman estimate is proven in \cite[Theorem 9.4.1]{KL} for the case
when boundary conditions (\ref{4.3}) are replaced with $u\mid _{S_{T}}=0,$ $%
u_{x}\mid _{\Gamma _{T}}=0.$ However, it easily follows from that proof that
the same result is valid for boundary conditions (\ref{4.3}), see formulas
(9.100) and (9.103) of \cite{KL}. We also need an estimate of the Volterra
integral, in which the CWF (\ref{4.1}) is involved. Theorem 4.2 is proven in 
\cite[Lemma 3.1.1]{KL}.

\begin{theorem}[An estimate for the Volterra integral operator] \label{Theorem 4.2} {The following Carleman estimate is valid:}%
  \begin{eqnarray}
  &&\int\limits_{Q_{T}}\left( \int\limits_{T/2}^{t}f\left( \mathbf{x},\tau
  \right) d\tau \right) ^{2}\varphi _{\lambda }\left( \mathbf{x},t\right) d%
  \mathbf{x}dt\leq \frac{C}{\lambda }\int\limits_{Q_{T}}f^{2}\left( \mathbf{x}%
  ,t\right) \varphi _{\lambda }\left( \mathbf{x},t\right) d\mathbf{x}dt,
  \label{4.4} \\
  &&\hspace{3cm}\forall \lambda >0,\forall f\in L_{2}\left( Q_{T}\right) , 
  \notag
  \end{eqnarray}%
  {where the number }$C=C\left( T,\alpha \right) >0${\ depends only
  on listed parameters.}
  \end{theorem}

\section{Numerical Method for Problem (\protect\ref{3.11})-(\protect\ref%
{3.14})}

\label{sec:5}

\subsection{Sets}

\label{sec:5.1}

The smoothness requirement for $W,$ which our method needs, is $W\in
C_{6}^{2}\left( \overline{Q}_{T}\right) .$ On the other hand, by (\ref{2.8}%
), (\ref{3.4}) and (\ref{3.11}) $W\in C_{6}^{4,2}\left( \overline{Q}%
_{T}\right) \subset C_{6}^{2}\left( \overline{Q}_{T}\right) .$ By the
embedding theorem 
\begin{equation}
H_{6}^{4}\left( Q_{T}\right) \subset C_{6}^{2}\left( \overline{Q}_{T}\right)
,\mbox{
	}\left\Vert u\right\Vert _{C_{6}^{2}\left( \overline{Q}_{T}\right) }\leq
C\left\Vert u\right\Vert _{H_{6}^{4}\left( Q_{T}\right) },\ \forall u\in
H_{6}^{4}\left( Q_{T}\right) ,  \label{5.1}
\end{equation}%
and $H_{6}^{4}\left( Q_{T}\right) $ is dense in $C_{6}^{2}\left( \overline{Q}%
_{T}\right) $ and compactly embedded in $C_{6}^{2}\left( \overline{Q}%
_{T}\right) $ in terms of the norm of the space $C_{6}^{2}\left( \overline{Q}%
_{T}\right) .$ The number $C=C\left( Q_{T}\right) $ in (\ref{5.1})depends
only on the domain $Q_{T}.$ We define the subspace $H_{6,0}^{4}\left(
Q_{T}\right) $ of the space $H_{6}^{4}\left( Q_{T}\right) $ as:%
\begin{equation}
H_{6,0}^{4}\left( Q_{T}\right) =\left\{ W\in H_{6}^{4}\left( Q_{T}\right)
:\partial _{n}W\mid _{S_{T}}=0,W\mid _{\Gamma _{T}}=0\right\} .  \label{6.1}
\end{equation}%
Let $M>0$ be an arbitrary number. Introduce two sets $B\left( M\right) $ and 
$B_{0}\left( M\right) ,$ 
\begin{equation}
B\left( M\right) =\left\{ 
\begin{array}{c}
W\in H_{6}^{4}\left( Q_{T}\right) :\text{ }\left\Vert W\right\Vert
_{H_{6}^{4}\left( Q_{T}\right) }<M, \\ 
W\mid _{\Gamma _{T}}=G_{0}\left( y,t\right) ,\partial _{n}W\mid
_{S_{T}}=G_{1}\left( \mathbf{x},t\right) ,\ 
\end{array}%
\right\} ,  \label{5.2}
\end{equation}%
\begin{equation}
B_{0}\left( M\right) =\left\{ W\in H_{6,0}^{4}\left( Q_{T}\right)
:\left\Vert W\right\Vert _{H_{6}^{4}\left( Q_{T}\right) }<M\right\} ,
\label{5.3}
\end{equation}%
where vector functions $G_{0},G_{1}$ were defined in (\ref{3.14}). By (\ref%
{5.1}), (\ref{5.2}) and (\ref{5.3}) 
\begin{equation}
\left\Vert W\right\Vert _{C_{6}^{2}\left( \overline{Q}_{T}\right) }\leq CM,%
\text{ }\forall W\in B\left( M\right) \cup B_{0}\left( M\right) .
\label{5.4}
\end{equation}%
We assume that in (\ref{3.90})%
\begin{equation}
\left\Vert S\right\Vert _{C_{4}\left( \overline{\Omega }\right) }<M.
\label{5.40}
\end{equation}%
It follows from (\ref{3.15}), (\ref{5.4}) and (\ref{5.40}) that the
following holds true for the vector function $Y$ in (\ref{3.13}):%
\begin{equation}
\left. 
\begin{array}{c}
\left\vert Y\left( W_{1},S_{1}\right) -Y\left( W_{2},S_{2}\right)
\right\vert \leq \\ 
\leq C_{1}\left( \left\vert W_{1}-W_{2}\right\vert \left( \mathbf{x}%
,t\right) +\left\vert \nabla W_{1}-\nabla W_{2}\right\vert \left( \mathbf{x}%
,t\right) \right) + \\ 
+C_{1}\left( \left\vert \int\limits_{T/2}^{t}\left\vert
W_{1}-W_{2}\right\vert \left( \mathbf{x},\tau \right) d\tau \right\vert
+\left\vert \int\limits_{T/2}^{t}\left\vert \nabla W_{1}-\nabla
W_{2}\right\vert \left( \mathbf{x},\tau \right) d\tau \right\vert \right) +
\\ 
+C_{1}\left\vert S_{1}\left( \mathbf{x}\right) -S_{2}\left( \mathbf{x}%
\right) \right\vert , \\ 
\forall W_{1},W_{2}\in \left( B\left( M\right) \cup B_{0}\left( M\right)
\right) ,\text{ }\forall \left( \mathbf{x},t\right) \in Q_{T},%
\end{array}%
\right.  \label{5.5}
\end{equation}%
where vector functions $S_{1},S_{2}$ satisfy (\ref{5.40}). Here and
everywhere below, \newline
$C_{1}=C_{1}\left( M,T,\Omega \right) >0$ denotes different numbers
depending only on $M,T,\Omega $.

\subsection{The sequence of quadratic functionals}

\label{sec:5.2}

We construct in this subsection 5.2 a sequence of weighted quadratic
Tikhonov-like functionals with the CWF (\ref{4.1}) in them to be minimized
on the set $\overline{B\left( M\right) }$.

\begin{remark}
  \label{remark 5.1}
  The paper \cite{Le} works with a version of the CCMM for a CIP for a 1d hyperbolic PDE. Similarly with the current paper, the minimizations of weighted quadratic functionals, although with a different CWF in them, are performed in \cite{Le} on a bounded set. However, that set is significantly different from our set $\overline{B(M)}$ due to a significant difference between our CIP and the CIP of \cite{Le}. The latter difference, in turn causes a significant difference in the convergence analysis of these two papers.
\end{remark}

\subsubsection{The functional number 0}

\label{sec:5.2.1}

Consider only the linear part of operators in (\ref{3.12}). First, we
introduce the linear operator $L,$%
\begin{equation}
\left. 
\begin{array}{c}
L\left( W\right) =\partial _{t}W-d\Delta W+ \\ 
+\left( \mathop{\rm div}\nolimits\left( v_{1}q_{S}\right) ,\mathop{\rm div}%
\nolimits\left( v_{2}q_{I}\right) ,\mathop{\rm div}\nolimits\left(
v_{3}q_{I}\right) ,\mathop{\rm div}\nolimits\left( v_{1t}q_{S}\right) ,%
\mathop{\rm div}\nolimits\left( v_{2t}q_{I}\right) ,\mathop{\rm div}%
\nolimits\left( v_{3t}q_{I}\right) \right) ^{T}.%
\end{array}%
\right.  \label{5.06}
\end{equation}%
Next, we construct the functional number zero as: 
\begin{equation}
\left. 
\begin{array}{c}
J_{0,\lambda ,\xi }:\overline{B\left( M\right) }\rightarrow \mathbb{R}, \\ 
J_{0,\lambda ,\xi }\left( W\right) =e^{-2\lambda b^{2}}\int\limits_{Q_{T}} 
\left[ L\left( W\right) \right] ^{2}\varphi _{\lambda }d\mathbf{x}dt+\xi
\left\Vert W\right\Vert _{H_{6}^{4}\left( Q_{T}\right) }^{2}.%
\end{array}%
\right.  \label{5.6}
\end{equation}%
Here $\xi \in \left( 0,1\right) $ is the regularization parameter and $\xi
\left\Vert W_{0}\right\Vert _{H_{3}^{4}\left( Q_{T}\right) }^{2}$ is the
Tikhonov regularization term. We recall that in the regularization theory,
this term is always considered in the norm of such a Banach space, which is
dense in the original space in terms of the norm of the original space,
compactly embedded in the original one and its norm is stronger than the
norm of the original space \cite{T}. The original space in our case is $%
C_{6}^{2}\left( \overline{Q}_{T}\right) ,$ also, see (\ref{5.1}). The
multiplier $e^{-2\lambda b^{2}}$ in (\ref{5.6}) balances integral term with
the regularization term since (\ref{4.2}) implies that $\max_{\overline{Q}%
_{T}}\varphi _{\lambda }\left( \mathbf{x},t\right) =e^{-2\lambda b^{2}}.$

\textbf{Minimization Problem Number Zero.} {Minimize the functional }$%
J_{0,\lambda ,\xi }\left( W_{0}\right) ${\ on the set }$\overline{B\left(
M\right) }.$

\subsubsection{The functional number $n\geq 1$}

\label{sec:5.2.2}

Suppose that for an appropriate given values of $\lambda $ and $\xi $ we
have constructed already vector functions 
\begin{equation}
W_{0,\min ,\lambda ,\xi },W_{1,\min ,\lambda ,\xi },...,W_{n-1,\min ,\lambda
,\xi }\in \overline{B\left( M\right) },  \label{5.7}
\end{equation}%
where $W_{k,\min ,\lambda ,\xi }$ is the minimizer of a certain functional $%
J_{k,\lambda ,\xi }\left( W\right) $ on the set $\overline{B\left( M\right) }%
:$ we will prove below that such a minimizer exists and is unique. Let $Y$
be the 6D vector function in (\ref{3.12}), (\ref{3.13}), (\ref{3.15}) and $L$
be the linear operator in (\ref{5.06}). Consider the following functional $%
J_{n,\lambda ,\xi }$: 
\begin{equation}
\left. 
\begin{array}{c}
J_{n,\lambda ,\xi }:\overline{B\left( M\right) }\rightarrow \mathbb{R}, \\ 
J_{n,\lambda ,\xi }\left( W\right) \mathbb{=}e^{-2\lambda
b^{2}}\int\limits_{Q_{T}}\left[ L\left( W\right) +Y\left( W_{n-1,\min
,\lambda },S\right) \right] ^{2}\mathbb{\varphi }_{\lambda }d\mathbf{x}dt%
\mathbb{+} \\ 
+\mathbb{\xi }\left\Vert W\right\Vert _{H_{6}^{4}\left( Q_{T}\right) }^{2}.%
\end{array}%
\right.  \label{5.8}
\end{equation}

\textbf{Minimization Problem Number }$n$\textbf{.} {Minimize the functional }%
$J_{n,\lambda ,\xi }\left( W\right) ${\ on the set }$\overline{B\left(
M\right) }.$

\section{Some Properties of the Functional $J_{n,\protect\lambda ,\protect%
\xi }\left( W\right) $}

\label{sec:6}

Since the functional \newline
$J_{n,\lambda ,\xi }\left( W\right) $ is quadratic, it is automatically
strongly convex on the entire space $H_{6}^{4}\left( Q_{T}\right) $ due to
the presence of the regularization term. However, in our convergence
analysis, we need to see how the presence of the terms with $\nabla W$ and $%
W $ affects the strong convexity estimate on the set $\overline{B\left(
M\right) }$ and how this is linked with the regularization parameter $\xi .$
This can be done using the Carleman estimate of Theorem 4.1. Besides, we
establish in this section the existence and uniqueness of the minimizer on
the set $\overline{B\left( M\right) }$ of the functional $J_{n,\lambda ,\xi
}\left( W\right) .$ In addition, we establish here the global convergence to
that minimizer of two versions of the gradient method of the minimization of 
$J_{n,\lambda ,\xi }\left( W\right) .$

We note that the Riesz theorem is inapplicable here since we search for the
minimizer on the bounded set $\overline{B\left( M\right) }\subset
H_{6}^{4}\left( Q_{T}\right) $ rather than on the whole space $%
H_{6}^{4}\left( Q_{T}\right) .$

\subsection{Strong convexity of the functional $J_{n,\protect\lambda ,%
\protect\xi }\left( W\right) $ on the set $\overline{B\left( M\right) }$}

\label{sec:6.1}

\begin{theorem}\label{Theorem 6.1} {Let the vector function }$W_{n-1,\min ,\lambda
  }\in \overline{B\left( M\right) }$ {and let (\ref{5.40}) holds.} {%
  The following three assertions are true:}
  
  {1. For each value of }$\lambda >0${\ and for each }$W\in 
  \overline{B\left( M\right) }${\ the functional }$J_{n,\lambda ,\xi
  }\left( W\right) ${\ has the Fr\'{e}chet derivative }$J_{n,\lambda ,\xi
  }^{\prime }\left( W\right) \in H_{6,0}^{4}\left( Q_{T}\right) ${. The Fr%
  \'{e}chet derivative has the form}%
  \begin{equation}
  \left. 
  \begin{array}{c}
  J_{n,\lambda ,\xi }^{\prime }\left( W\right) \left( h\right) =2e^{-2\lambda
  b^{2}}\int\limits_{Q_{T}}\left[ L\left( W\right) +Y\left( W_{n-1,\min
  ,\lambda },S\right) \right] L\left( h\right) \varphi _{\lambda }d\mathbf{x}%
  dt+ \\ 
  +2\xi \left( h,W\right) ,\text{ }\forall h\in H_{6,0}^{4}\left( Q_{T}\right)
  .%
  \end{array}%
  \right.   \label{6.2}
  \end{equation}%
  {where }$\left( ,\right) ${\ is the scalar product in }$%
  H_{6}^{4}\left( Q_{T}\right) .${\ Furthermore, }$J_{n,\lambda ,\xi
  }^{\prime }\left( W\right) ${\ is Lipschitz continuous on }$\overline{%
  B\left( M\right) }.${\ This means that there exists a number }$%
  D=D\left( \lambda ,\xi ,M\right) >0${\ depending only on listed
  parameters such that }%
  \begin{equation}
  \left. 
  \begin{array}{c}
  \left\Vert J_{n,\lambda ,\xi }^{\prime }\left( W_{1}\right) -J_{n,\lambda
  ,\xi }^{\prime }\left( W_{2}\right) \right\Vert _{H_{6}^{4}\left(
  Q_{T}\right) }\leq D\left\Vert W_{1}-W_{2}\right\Vert _{H_{6}^{4}\left(
  Q_{T}\right) },\text{ } \\ 
  \forall W_{1},W_{2}\in \overline{B\left( M\right) }.%
  \end{array}%
  \right.   \label{6.3}
  \end{equation}
  
  {2. Let }$\lambda _{0}=\lambda _{0}\left( Q_{T},d\right) \geq 1${\
  be the number of Theorem 4.1. There exists a sufficiently large number }$%
  \lambda _{1}=\lambda _{1}\left( Q_{T},d,M\right) \geq \lambda _{0}${\
  such that if }$\lambda \geq \lambda _{1}$ {and the regularization
  parameter }$\xi ${\ is such that}%
  \begin{equation}
  \frac{\xi }{2}\in \left[ \exp \left( -\lambda \frac{T^{2}}{4}\right) ,\frac{1%
  }{2}\right) ,  \label{6.4}
  \end{equation}%
  {then the functional }$J_{n,\lambda ,\xi }\left( W\right) ${\
  satisfies the following estimate:}%
  \begin{equation}
  \left. 
  \begin{array}{c}
  J_{n,\lambda ,\xi }\left( W_{2}\right) -J_{n,\lambda ,\xi }\left(
  W_{1}\right) -J_{n,\lambda ,\xi }^{\prime }\left( W_{1}\right) \left(
  W_{2}-W_{1}\right) \geq \\ 
  \geq C_{1}\lambda \dint\limits_{Q_{T}}\left[ \left( \nabla \left(
  W_{2}-W_{1}\right) \right) ^{2}+\left( W_{2}-W_{1}\right) ^{2}\right]
  \varphi _{\lambda }\left( \mathbf{x},t\right) d\mathbf{x}dt+ \\ 
  +\left( \xi /2\right) \left\Vert W_{2}-W_{1}\right\Vert _{H_{6}^{4}\left(
  Q_{T}\right) }^{2},\text{ } \\ 
  \forall W_{1},W_{2}\in \overline{B\left( M\right) },\ \forall \lambda \geq
  \lambda _{1},%
  \end{array}%
  \right.  \label{6.5}
  \end{equation}%
  {i.e. this functional is strongly convex on the set }$\overline{B\left(
  M\right) }.$
  
  {3.} {For each }$\lambda \geq \lambda _{1}${\ there exists
  unique minimizer }$W_{n,\min ,\lambda }\in \overline{B\left( M\right) }$%
  {\ of the functional }$J_{n,\lambda ,\xi }\left( W\right) ${\ on
  the set }$\overline{B\left( M\right) }${\ and the following inequality
  holds:}%
  \begin{equation}
  J_{n,\lambda ,\xi }^{\prime }\left( W_{n,\min ,\lambda ,\xi }\right) \left(
  W-W_{n,\min ,\lambda ,\xi }\right) \geq 0,\ \forall W\in \overline{B\left(
  M\right) }.  \label{6.6}
  \end{equation}
\end{theorem}

\begin{remark} \label{Remark 6.1} Thus, the right hand side of the strong convexity
estimate (\ref{6.4}) includes not only the quadratic term $\left( \xi
/2\right) \left\Vert W_{2}-W_{1}\right\Vert _{H_{3}^{4}\left( Q_{T}\right)
}^{2},${\ which is to be expected since the functional }$J_{n,\lambda
,\xi }${\ is quadratic, but other terms as well. The latter is
important in our convergence analysis below.}
\end{remark}

\begin{proof}[Proof of \Cref{Theorem 6.1}] Let $W_{1},W_{2}\in \overline{B\left(
  M\right) }$ be two arbitrary points. Denote $h=W_{2}-W_{1}$. Then $%
  W_{2}=W_{1}+h.$ By (\ref{6.1})-(\ref{5.3}) and triangle inequality%
  \begin{equation}
  h\in B_{0}\left( 2M\right) .  \label{9.1}
  \end{equation}%
  By (\ref{5.8})%
  \begin{equation}
  \left. 
  \begin{array}{c}
  J_{n,\lambda ,\xi }\left( W_{1}+h\right) \mathbb{=}e^{-2\lambda
  b^{2}}\int\limits_{Q_{T}}\left[ L\left( W_{1}\right) +L\left( h\right)
  +Y\left( W_{n-1,\min ,\lambda },S\right) \right] ^{2}\mathbb{\varphi }%
  _{\lambda }d\mathbf{x}dt\mathbb{+} \\ 
  +\mathbb{\xi }\left\Vert W_{1}+h\right\Vert _{H_{6}^{4}\left( Q_{T}\right)
  }^{2}= \\ 
  =e^{-2\lambda b^{2}}\int\limits_{Q_{T}}\left( L\left( W_{1}\right) +Y\left(
  W_{n-1,\min ,\lambda },S\right) \right) ^{2}\mathbb{\varphi }_{\lambda }d%
  \mathbf{x}dt\mathbb{+\xi }\left\Vert W_{1}\right\Vert _{H_{6}^{4}\left(
  Q_{T}\right) }^{2} \\ 
  +2e^{-2\lambda b^{2}}\int\limits_{Q_{T}}\left( L\left( W_{1}\right) +Y\left(
  W_{n-1,\min ,\lambda },S\right) \right) L\left( h\right) \mathbb{\varphi }%
  _{\lambda }d\mathbf{x}dt+2\left( W_{1},h\right) + \\ 
  +e^{-2\lambda b^{2}}\int\limits_{Q_{T}}\left( L\left( h\right) \right) ^{2}%
  \mathbb{\varphi }_{\lambda }d\mathbf{x}dt+\xi \left\Vert h\right\Vert
  _{H_{6}^{4}\left( Q_{T}\right) }^{2}.%
  \end{array}%
  \right.   \label{9.2}
  \end{equation}%
  By (\ref{6.1}), (\ref{5.3}) and (\ref{9.1}) $h\in H_{6,0}^{4}\left(
  Q_{T}\right) .$ Consider the functional $Z_{n,\lambda ,\xi }\left( h\right) $
  defined as:%
  \begin{equation}
  Z_{n,\lambda ,\xi }\left( h\right) =2e^{-2\lambda
  b^{2}}\int\limits_{Q_{T}}\left( L\left( W_{1}\right) +Y\left( W_{n-1,\min
  ,\lambda },S\right) \right) L\left( h\right) \mathbb{\varphi }_{\lambda }d%
  \mathbf{x}dt+2\left( W_{1},h\right) .  \label{9.3}
  \end{equation}%
  Clearly, this is a linear bounded functional $Z_{n,\lambda ,\xi }:$ $%
  H_{6,0}^{4}\left( Q_{T}\right) \rightarrow \mathbb{R}.$ Hence, by Riesz
  theorem there exists unique point $\widehat{Z}_{n,\lambda ,\xi }\in
  H_{6,0}^{4}\left( Q_{T}\right) $ such that 
  \begin{equation}
  Z_{n,\lambda ,\xi }\left( h\right) =\left( \widehat{Z}_{n,\lambda ,\xi
  },h\right) ,\text{ }\forall h\in H_{6,0}^{4}\left( Q_{T}\right) .
  \label{9.4}
  \end{equation}%
  On the other hand, (\ref{9.2})-(\ref{9.4}) imply that 
  \begin{equation*}
  \lim_{\left\Vert h\right\Vert _{H_{6,}^{4}\left( Q_{T}\right) }\rightarrow 0}%
  \frac{J_{n,\lambda ,\xi }\left( W_{1}+h\right) -J_{n,\lambda ,\xi }\left(
  W_{1}\right) -\left( \widehat{Z}_{n,\lambda ,\xi },h\right) }{\left\Vert
  h\right\Vert _{H_{6,}^{4}\left( Q_{T}\right) }}=0.
  \end{equation*}%
  Hence, $\widehat{Z}_{n,\lambda ,\xi }$ is the Fr\'{e}chet derivative of the
  functional $J_{n,\lambda ,\xi }\left( W_{1}\right) $ at the point $W_{1},$ 
  \begin{equation}
  \widehat{Z}_{n,\lambda ,\xi }=J_{n,\lambda ,\xi }^{\prime }\left(
  W_{1}\right) \in H_{6,0}^{4}\left( Q_{T}\right) .  \label{9.5}
  \end{equation}%
  We omit the proof of the Lipschitz continuity property (\ref{6.3}) since
  this proof is quite similar with the proof of Theorem 5.3.1 of \cite{KL}, so
  as of Theorem 3.1 of \cite{Bak}.
  
  We now prove the strong convexity estimate (\ref{6.5}). By (\ref{9.2})-(\ref%
  {9.5})%
  \begin{equation}
  \left. 
  \begin{array}{c}
  J_{n,\lambda ,\xi }\left( W_{1}+h\right) -J_{n,\lambda ,\xi }\left(
  W_{1}\right) -J_{n,\lambda ,\xi }^{\prime }\left( W_{1}\right) \left(
  h\right) = \\ 
  =e^{-2\lambda b^{2}}\int\limits_{Q_{T}}\left( L\left( h\right) \right) ^{2}%
  \mathbb{\varphi }_{\lambda }d\mathbf{x}dt+\xi \left\Vert h\right\Vert
  _{H_{6}^{4}\left( Q_{T}\right) }^{2}.%
  \end{array}%
  \right.   \label{9.6}
  \end{equation}%
  Using Carleman estimate (\ref{4.30}), estimate from the below the first term
  in the second line of (\ref{9.6}). By (\ref{2.03}) and (\ref{5.06})%
  \begin{equation}
  \left( L\left( h\right) \right) ^{2}\geq \left( h_{t}-d\Delta h\right)
  ^{2}-C_{1}\left( \left\vert \nabla h\right\vert ^{2}+\left\vert h\right\vert
  ^{2}\right) .  \label{9.7}
  \end{equation}%
  Since $h\in H_{6,0}^{4}\left( Q_{T}\right) ,$ then by (\ref{4.3}) and (\ref%
  {6.1}) we can apply Carleman estimate (\ref{4.30}) to the vector function $h$%
  . Hence, setting $\lambda \geq \lambda _{0},$ multiplying (\ref{9.7}) by $%
  \mathbb{\varphi }_{\lambda }e^{-2\lambda b^{2}}$ and integrating over $Q_{T},
  $ we obtain%
  \begin{equation}
  \left. 
  \begin{array}{c}
  e^{-2\lambda b^{2}}\int\limits_{Q_{T}}\left( L\left( h\right) \right) ^{2}%
  \mathbb{\varphi }_{\lambda }d\mathbf{x}dt\geq e^{-2\lambda
  b^{2}}\int\limits_{Q_{T}}\left( h_{t}-d\Delta h\right) ^{2}\mathbb{\varphi }%
  _{\lambda }d\mathbf{x}dt- \\ 
  -e^{-2\lambda b^{2}}C_{1}\int\limits_{Q_{T}}\left( \left\vert \nabla
  h\right\vert ^{2}+\left\vert h\right\vert ^{2}\right) \mathbb{\varphi }%
  _{\lambda }d\mathbf{x}dt\geq  \\ 
  \geq Ce^{-2\lambda b^{2}}\int\limits_{Q_{T}}\left( \lambda \left\vert \nabla
  h\right\vert ^{2}+\lambda ^{3}\left\vert h\right\vert ^{2}\right) \varphi
  _{\lambda }d\mathbf{x}dt- \\ 
  -e^{-2\lambda b^{2}}C_{1}\int\limits_{Q_{T}}\left( \left\vert \nabla
  h\right\vert ^{2}+\left\vert h\right\vert ^{2}\right) \mathbb{\varphi }%
  _{\lambda }d\mathbf{x}dt- \\ 
  -C\left( \left\Vert h\left( \mathbf{x},T\right) \right\Vert
  _{H_{6}^{1}\left( \Omega \right) }^{2}+\left\Vert h\left( \mathbf{x}%
  ,0\right) \right\Vert _{H_{6}^{1}\left( \Omega \right) }^{2}\right) \lambda
  ^{2}\exp \left( -\lambda T^{2}/2\right) 
  \end{array}%
  \right.   \label{9.8}
  \end{equation}%
  Choose $\lambda _{1}=\lambda _{1}\left( Q_{T},d,M\right) \geq \lambda _{0}$
  so large that $C\lambda _{1}>2C_{1}.$ Then the second term in the third line
  of (\ref{9.8}) is absorbed by the first term in that line. Hence, (\ref{9.8}%
  ) becomes%
  \begin{equation}
  \left. 
  \begin{array}{c}
  e^{-2\lambda b^{2}}\int\limits_{Q_{T}}\left( L\left( h\right) \right) ^{2}%
  \mathbb{\varphi }_{\lambda }d\mathbf{x}dt\geq C_{1}\lambda
  \int\limits_{Q_{T}}\left( \left\vert \nabla h\right\vert ^{2}+\left\vert
  h\right\vert ^{2}\right) \varphi _{\lambda }d\mathbf{x}dt- \\ 
  -C\left( \left\Vert h\left( \mathbf{x},T\right) \right\Vert
  _{H_{6}^{1}\left( \Omega \right) }^{2}+\left\Vert h\left( \mathbf{x}%
  ,0\right) \right\Vert _{H_{6}^{1}\left( \Omega \right) }^{2}\right) \lambda
  ^{2}\exp \left( -\lambda T^{2}/2\right) ,\text{ }\forall \lambda \geq
  \lambda _{1}.%
  \end{array}%
  \right.   \label{9.9}
  \end{equation}%
  Now, by (\ref{5.3}), (\ref{9.1}) and trace theorem%
  \begin{equation*}
  \left\Vert h\left( \mathbf{x},T\right) \right\Vert _{H_{6}^{1}\left( \Omega
  \right) }^{2}+\left\Vert h\left( \mathbf{x},0\right) \right\Vert
  _{H_{6}^{1}\left( \Omega \right) }^{2}\leq C_{1}\left\Vert h\right\Vert
  _{H_{6}^{4}\left( Q_{T}\right) }^{2}.
  \end{equation*}%
  Hence, by (\ref{5.4}) 
  \begin{equation}
  \left. 
  \begin{array}{c}
  \xi \left\Vert h\right\Vert _{H_{6}^{4}\left( Q_{T}\right) }^{2}-C\left(
  \left\Vert h\left( \mathbf{x},T\right) \right\Vert _{H_{6}^{1}\left( \Omega
  \right) }^{2}+\left\Vert h\left( \mathbf{x},0\right) \right\Vert
  _{H_{6}^{1}\left( \Omega \right) }^{2}\right) \lambda ^{2}\exp \left(
  -\lambda T^{2}/2\right) \geq  \\ 
  \geq \left( \xi -C_{1}\lambda ^{2}\exp \left( -\lambda T^{2}/2\right)
  \right) \left\Vert h\right\Vert _{H_{6}^{4}\left( Q_{T}\right) }^{2}\geq
  \left( \xi /2\right) \left\Vert h\right\Vert _{H_{6}^{4}\left( Q_{T}\right)
  }^{2}.%
  \end{array}%
  \right.   \label{9.10}
  \end{equation}%
  Thus, using (\ref{9.6}), (\ref{9.9}) and (\ref{9.10}), we obtain 
  \begin{equation}
  \left. 
  \begin{array}{c}
  J_{n,\lambda ,\xi }\left( W_{1}+h\right) -J_{n,\lambda ,\xi }\left(
  W_{1}\right) -J_{n,\lambda ,\xi }^{\prime }\left( W_{1}\right) \left(
  h\right) \geq  \\ 
  \geq C_{1}\lambda \int\limits_{Q_{T}}\left( \left\vert \nabla h\right\vert
  ^{2}+\left\vert h\right\vert ^{2}\right) \varphi _{\lambda }d\mathbf{x}%
  dt+\left( \xi /2\right) \left\Vert h\right\Vert _{H_{6}^{4}\left(
  Q_{T}\right) }^{2},\text{ }\forall \lambda \geq \lambda _{1},%
  \end{array}%
  \right.   \label{9.11}
  \end{equation}%
  which proves (\ref{6.5}). As soon as (\ref{6.5}) is established, existence
  and uniqueness of the minimizer $W_{n,\min ,\lambda ,\xi }$ of the
  functional $J_{n,\lambda ,\xi }$ on the set $\overline{B\left( M\right) }$
  as well as inequality (\ref{6.6}) follow from a simple combination of Lemma
  2.1 with Theorem 2.1 of \cite{Bak} as well as from a combination of Lemma
  5.2.1 with Theorem 5.2.1 of \cite{KL}.
\end{proof}

\subsection{Global convergence of the gradient projection method}

\label{sec:6.2}

Suppose that there exists a vector function 
\begin{equation}
F\in B\left( M\right) .  \label{6.07}
\end{equation}%
For each vector function $W\in B\left( M\right) $ consider the difference%
\begin{equation}
\widetilde{W}=W-F.  \label{6.7}
\end{equation}%
By (\ref{5.2}), (\ref{5.3}), (\ref{6.7}) and triangle inequality 
\begin{equation}
\widetilde{W}\in B_{0}\left( 2M\right) ,\text{ }\forall W\in B\left(
M\right) .  \label{6.8}
\end{equation}%
The convenience of the transformation (\ref{6.7}) is that since $B_{0}\left(
2M\right) $ is the ball of the radius $2M$ with the center at $\left\{
0\right\} $ in the space $H_{6,0}^{4}\left( Q_{T}\right) $, then it is easy
to construct projection operator of $H_{6,0}^{4}\left( Q_{T}\right) $ on $%
\overline{B_{0}\left( 2M\right) }$ as:%
\begin{equation}
P_{B_{0}\left( 2M\right) }\left( Y\right) =\left\{ 
\begin{array}{c}
Y\text{ if }Y\in \overline{B_{0}\left( 2M\right) }, \\ 
2M\cdot Y/\left\Vert Y\right\Vert _{H_{6}^{4}\left( Q_{T}\right) }\text{ if }%
Y\notin \overline{B_{0}\left( 2M\right) }.%
\end{array}%
\right.  \label{6.9}
\end{equation}%
By (\ref{6.07}) and triangle inequality 
\begin{equation}
\widetilde{W}+F\in B\left( 3M\right) ,\text{ }\forall \widetilde{W}\in 
\overline{B_{0}\left( 2M\right) }  \label{6.10}
\end{equation}%
Consider the functional $I_{n,\lambda ,\xi }:\overline{B_{0}\left( 2M\right) 
}\rightarrow \mathbb{R}$ defined as:%
\begin{equation}
\left. 
\begin{array}{c}
I_{n,\lambda ,\xi }\left( \widetilde{W}\right) =J_{n,\lambda ,\xi }\left( 
\widetilde{W}+F\right) = \\ 
=e^{-2\lambda b^{2}}\int\limits_{Q_{T}}\left[ L\left( \widetilde{W}+F\right)
+Y\left( \widetilde{W}_{n-1,\min ,\lambda }+F,S\right) \right] ^{2}\mathbb{%
\varphi }_{\lambda }d\mathbf{x}dt+ \\ 
+\mathbb{\xi }\left\Vert \widetilde{W}+F\right\Vert _{H_{6}^{4}\left(
Q_{T}\right) }^{2},\text{ }\forall \widetilde{W}\in \overline{B_{0}\left(
2M\right) },%
\end{array}%
\right.  \label{6.110}
\end{equation}%
where $\widetilde{W}_{n-1,\min ,\lambda }\in \overline{B_{0}\left( 2M\right) 
}$ is the unique minimizer of the functional \newline
$J_{n-1,\lambda ,\xi }\left( \widetilde{W}+F\right) $ on the set $\overline{%
B_{0}\left( 2M\right) }.$ An obvious analog of Theorem 6.1 is valid for the
functional $I_{n,\lambda ,\xi }\left( \widetilde{W}\right) $, where by (\ref%
{6.10})%
\begin{equation}
\lambda \geq \lambda _{2}=\lambda _{1}\left( Q_{T},d,3M\right) \geq \lambda
_{1}\left( Q_{T},d,M\right) .  \label{6.12}
\end{equation}

Theorem 6.1 and (\ref{6.12}) justify the existence and uniqueness of the
minimizer $\widetilde{W}_{n,\min ,\lambda }\in \overline{B_{0}\left(
2M\right) }$ of the functional $I_{n,\lambda ,\xi }$ on the set $\overline{%
B_{0}\left( 2M\right) }$ for any $n$ and for $\lambda $ satisfying (\ref%
{6.12}), which, in turn justifies the presence of the term $\widetilde{W}%
_{n-1,\min ,\lambda }$ in (\ref{6.110}). We now construct a sequence
converging to $\widetilde{W}_{\min ,n,\lambda ,\xi }.$ We start from an
arbitrary point $\widetilde{W}_{0}\in B_{0}\left( 2M\right) .$ The sequence
of the gradient projection method is: 
\begin{equation}
\widetilde{W}_{n}=P_{B_{0}\left( 2M\right) }\left( \widetilde{W}%
_{n-1}-\gamma I_{n,\lambda ,\xi }^{\prime }\left( \widetilde{W}_{n-1}\right)
\right) ,\text{ }n=1,2,...,  \label{6.13}
\end{equation}%
where the step size $\gamma >0$ is a certain number and $I_{n,\lambda ,\xi
}^{\prime }\left( \widetilde{W}_{n-1}\right) $ is the Fr\'{e}chet derivative
of the functional $I_{n,\lambda ,\xi }\left( \widetilde{W}\right) $ at the
point $\widetilde{W}_{n-1}$. \ Note that since by Theorem 6.1 $I_{n,\lambda
,\xi }^{\prime }\in H_{6,0}^{4}\left( Q_{T}\right) $ and since (\ref{6.9})
holds, then in (\ref{6.13}) $\widetilde{W}_{n}\in H_{6,0}^{4}\left(
Q_{T}\right) $ for all $n=1,2,\dots$. \Cref{Theorem 6.2} follows immediately
from a combination of Theorem 6.1 with Theorem 2.1 of \cite{Bak}.

\begin{theorem}\label{Theorem 6.2} {Let parameters }$\lambda ${\ and }$\xi $%
  {\ satisfy (\ref{6.12}) and (\ref{6.4}) respectively. Then there exists
  a sufficiently small number }$\gamma _{0}>0${\ such that for any }$%
  \gamma \in \left( 0,\gamma _{0}\right) $ {the sequence (\ref{6.13})
  converges to the minimizer }$\widetilde{W}_{\min ,n,\lambda ,\xi }${.
  Furthermore, there exists a number }$\theta =\theta \left( \gamma \right)
  \in \left( 0,1\right) ${\ such that the following convergence estimate
  holds:}%
  \begin{equation*}
  \left\Vert \widetilde{W}_{n}-\widetilde{W}_{\min ,n,\lambda ,\xi
  }\right\Vert _{H_{6}^{4}\left( Q_{T}\right) }\leq \theta ^{n}\left\Vert 
  \widetilde{W}_{0}-\widetilde{W}_{\min ,n,\lambda ,\xi }\right\Vert
  _{H_{6}^{4}\left( Q_{T}\right) },\text{ }n=1,2,...
  \end{equation*}
\end{theorem}

Thus, the gradient projection method (\ref{6.13}) of the minimization of the
functional $I_{n,\lambda ,\xi }$ on the set $\overline{B_{0}\left( 2M\right) 
}$ converges globally, since it starts from an arbitrary point $\widetilde{W}%
_{0}\in B_{0}\left( 2M\right) $ and smallness condition is not imposed on $M$%
.

\subsection{Global convergence of the gradient descent method}

\label{sec:6.3}

The gradient projection method is not easy to implement due to the necessity
to obtain zero boundary conditions via (\ref{6.07}), (\ref{6.7}). In
addition, it is inconvenient to use the projection operator (\ref{6.9}) in
computations. Hence, we now formulate a simpler gradient descent method of
the minimization of functional (\ref{5.8}) on the set $\overline{B\left(
M\right) }.$ To do this, we need to assume that the minimizer that
functional belongs to the set $B\left( M/3\right) .$

Let $W_{0}\in B\left( M/3\right) $ be an arbitrary point and let $\gamma \in
\left( 0,1\right) $ be a number. The sequence of the gradient descent method
is:%
\begin{equation}
W_{n}=W_{n-1}-\gamma J_{n,\lambda ,\xi }^{\prime }\left( W_{n-1}\right) ,%
\text{ }n=1,2,...  \label{6.14}
\end{equation}

\Cref{Theorem 6.3} follows immediately from a combination of Theorem 6.1
with Theorem 6 of \cite{SAR}.

\begin{theorem}\label{Theorem 6.3} {Let }$\lambda \geq \lambda _{1},${\ where }$%
  \lambda _{1}${\ is the number of Theorem 6.1 and let (\ref{6.4}) holds.
  Let }$W_{n,\min ,\lambda ,\xi }\in \overline{B\left( M\right) }${\ be
  the minimizer of the functional }$J_{n,\lambda ,\xi }${\ on the set }$%
  \overline{B\left( M\right) },${\ which was found in Theorem 6.1. Assume
  that }$W_{n,\min ,\lambda ,\xi }\in B\left( M/3\right) .${\ Then there
  exists a sufficiently small number }$\gamma _{0}>0${\ such that for any 
  }$\gamma \in \left( 0,\gamma _{0}\right) ${\ all terms of sequence (\ref%
  {6.14}) belong to the set }$B\left( M\right) .${\ Furthermore, the
  sequence (\ref{6.14}) converges to }$W_{n,\min ,\lambda ,\xi }${\ and
  there exists a number }$\theta =\theta \left( \gamma \right) \in \left(
  0,1\right) ${\ such that the following convergence estimate is valid:} 
  \begin{equation*}
  \left\Vert W_{n}-W_{\min ,n,\lambda ,\xi }\right\Vert _{H_{6}^{4}\left(
  Q_{T}\right) }\leq \theta ^{n}\left\Vert W_{0}-W_{\min ,n,\lambda ,\xi
  }\right\Vert _{H_{6}^{4}\left( Q_{T}\right) },\text{ }n=1,2,...
  \end{equation*}
\end{theorem}

This is again the global convergence property since $W_{0}\in B\left(
M/3\right) $ is an arbitrary point and since smallness conditions are not
imposed on the number $M$.

\section{Convergence of Minimizers to the True Solution}

\label{sec:7}

One of the main concepts of the regularization theory is the assumption that
there exists an \textquotedblleft ideal" solution of an ill-posed problem
with the \textquotedblleft ideal", noiseless input data \cite{T}. We call
this \textquotedblleft exact solution". The input data in our case are
lateral data $G_{0},G_{1}$ in (\ref{3.14}) as well as the vector function $S$
in (\ref{3.90}).

The above Theorems 6.1-6.3 characterize some properties of the functionals $%
J_{n,\lambda ,\xi }.$ However, they do not address the question about the
relevance of minimizers of these functionals to the true solution of problem
(\ref{3.11})-(\ref{3.15}) as well as to the true solution of our CIP. These
key questions are addressed in Theorems 7.1-7.4 of this section.

\subsection{Introducing noise in the input data}

\label{sec:7.2}

Let $W^{\ast }$ be that exact solution of problem (\ref{3.12})-(\ref{3.14})
with the exact data $G_{0}^{\ast },G_{1}^{\ast }$ and the exact vector
function $S^{\ast },$ 
\begin{equation}
\left. 
\begin{array}{c}
L\left( W^{\ast }\right) +Y\left( W^{\ast },S^{\ast }\right) =0\text{ in }%
Q_{T}, \\ 
W^{\ast }\mid _{\Gamma _{T}}=G_{0}^{\ast }\left( y,t\right) ,\text{ }%
\partial _{n}W^{\ast }\mid _{S_{T}}=G_{1}^{\ast }\left( \mathbf{x},t\right) .%
\end{array}%
\right.  \label{7.1}
\end{equation}%
It is natural to assume that 
\begin{equation}
W^{\ast }\in B^{\ast }\left( M\right) ,  \label{7.01}
\end{equation}%
where, similarly with (\ref{5.2}), 
\begin{equation}
B^{\ast }\left( M\right) =\left\{ 
\begin{array}{c}
W\in H_{6}^{4}\left( Q_{T}\right) :\text{ }\left\Vert W\right\Vert
_{H_{6}^{4}\left( Q_{T}\right) }<M, \\ 
W\mid _{\Gamma _{T}}=G_{0}^{\ast }\left( y,t\right) ,\partial _{n}W\mid
_{S_{T}}=G_{1}^{\ast }\left( \mathbf{x},t\right) ,\ 
\end{array}%
\right\}  \label{7.2}
\end{equation}

Due to the requirement $u\in H_{0}^{2,1}\left( Q_{T}\right) $ in Carleman
estimate (\ref{4.30}), we need to work with the zero boundary conditions in
our convergence analysis. Hence, let $F\in B\left( M\right) $ be the vector
function in (\ref{6.07}). Let $\delta \in \left( 0,1\right) $ be the noise
level in the input data. Suppose that there exists a vector function $%
F^{\ast }\in B^{\ast }\left( M\right) $ such that 
\begin{equation}
\left\Vert F-F^{\ast }\right\Vert _{H_{6}^{4}\left( Q_{T}\right) }<\delta .
\label{7.3}
\end{equation}%
In addition, let%
\begin{equation}
\left\Vert S-S^{\ast }\right\Vert _{C_{4}\left( \overline{\Omega }\right)
}<\delta .  \label{7.5}
\end{equation}

Keep notations of subsection 6.2 and similarly with (\ref{6.7}) let 
\begin{equation}
\widetilde{W}^{\ast }=W^{\ast }-F^{\ast }.  \label{7.6}
\end{equation}%
Then, by (\ref{7.2}) and triangle inequality 
\begin{equation}
\widetilde{W}^{\ast }\in B_{0}\left( 2M\right) .  \label{7.7}
\end{equation}

\subsection{Convergence of minimizers}

\label{sec:7.3}

Denote%
\begin{equation}
\left. 
\begin{array}{c}
h_{n}=\widetilde{W}_{n,\min ,\lambda ,\xi }-\widetilde{W}^{\ast }, \\ 
s_{n}=W_{n,\min ,\lambda ,\xi }-W^{\ast }.%
\end{array}%
\right.  \label{7.8}
\end{equation}%
Recall that $W_{n,\min ,\lambda ,\xi }$ is the minimizer of the functional $%
J_{n,\lambda ,\xi }${\ }on the set $\overline{B\left( M\right) },$ and $%
\widetilde{W}_{n,\min ,\lambda ,\xi }$ is the minimizer of the functional $%
I_{n,\lambda ,\xi }$ in (\ref{6.110}) on the set $\overline{B_{0}\left(
2M\right) }.$ We cannot guarantee that $\widetilde{W}_{n,\min ,\lambda ,\xi
}+F=W_{n,\min ,\lambda ,\xi }.$ Hence, we provide two types of convergence
estimates: for noisy and noiseless data. The difference here is that we have
the same lateral Cauchy data for $W_{n,\min ,\lambda ,\xi }$ and $W^{\ast }$
in the case of noiseless data, and these data are not necessary zero. On the
other hand, in the noisy case, we have zero lateral Cauchy data for both $%
\widetilde{W}_{n,\min ,\lambda ,\xi }$ and $\widetilde{W}^{\ast }.$

The vector function $W^{\ast }\left( \mathbf{x},t\right) $ is actually
generated by the exact coefficients $\beta ^{\ast }\left( \mathbf{x}\right) $
and $\gamma ^{\ast }\left( \mathbf{x}\right) $ of the SIR system (\ref{2.2}%
)-(\ref{2.4}). We approximate them via the pair of functions $\beta
_{n,\lambda ,\delta }\left( \mathbf{x}\right) ,\gamma _{n,\lambda ,\delta
}\left( \mathbf{x}\right) $ in the case of noisy data and by the pair of
functions $\beta _{n,\lambda }\left( \mathbf{x}\right) ,\gamma _{n,\lambda
}\left( \mathbf{x}\right) $ in the case of noiseless data. Following the end
of section 3, we obtain these pairs of functions from the vector functions $%
\widetilde{W}_{n,\min ,\lambda ,\xi }+F$ in the case of noisy data and from
the vector functions $W_{n,\min ,\lambda ,\xi }$ for the noiseless cases
respectively via applying formulas (\ref{3.11}), (\ref{3.4}) and (\ref{3.6})
sequentially.

\begin{theorem}[convergence for noisy data]\label{Theorem 7.1} {Let conditions (%
  \ref{7.1})-(\ref{7.8}) hold. Let }$\lambda \geq \lambda _{2}${,} {%
  where }$\lambda _{2}\geq 1${\ is defined in (\ref{6.12}). Assume that
  condition (\ref{6.4}) holds. Then the following convergence estimates\ are
  valid:}%
  \begin{equation}
  \left. 
  \begin{array}{c}
  e^{-2\lambda b^{2}}\dint\limits_{Q_{T}}\left( \left\vert \nabla
  h_{n}\right\vert ^{2}+\left\vert h_{n}\right\vert ^{2}\right) \varphi
  _{\lambda }\left( \mathbf{x},t\right) d\mathbf{x}dt\leq  \\ 
  \leq \left( C_{1}/\lambda \right) ^{n}e^{-2\lambda
  b^{2}}\dint\limits_{Q_{T}}\left( \left\vert \nabla h_{0}\right\vert
  ^{2}+\left\vert h_{0}\right\vert ^{2}\right) \varphi _{\lambda }\left( 
  \mathbf{x},t\right) d\mathbf{x}dt+ \\ 
  +\left( C_{1}/\lambda \right) \left( \xi +\delta ^{2}\right) ,\text{ }%
  n=1,2,...,%
  \end{array}%
  \right.   \label{7.10}
  \end{equation}%
  \begin{equation}
  \left. 
  \begin{array}{c}
  e^{-2\lambda b^{2}}\dint\limits_{Q_{T}}\left\vert \nabla \left( \widetilde{W}%
  _{n,\min ,\lambda ,\xi }+F\right) -\nabla W^{\ast }\right\vert ^{2}\varphi
  _{\lambda }\left( \mathbf{x},t\right) d\mathbf{x}dt+ \\ 
  +e^{-2\lambda b^{2}}\dint\limits_{Q_{T}}\left\vert \left( \widetilde{W}%
  _{n,\min ,\lambda ,\xi }+F\right) -W^{\ast }\right\vert ^{2}\varphi
  _{\lambda }\left( \mathbf{x},t\right) d\mathbf{x}dt\leq  \\ 
  \leq \left( C_{1}/\lambda \right) ^{n}e^{-2\lambda
  b^{2}}\dint\limits_{Q_{T}}\left\vert \nabla \left( \widetilde{W}_{0,\min
  ,\lambda ,\xi }+F\right) -\nabla W^{\ast }\right\vert ^{2}\varphi _{\lambda
  }\left( \mathbf{x},t\right) d\mathbf{x}dt+ \\ 
  +\left( C_{1}/\lambda \right) ^{n}e^{-2\lambda
  b^{2}}\dint\limits_{Q_{T}}\left\vert \left( \widetilde{W}_{0,\min ,\lambda
  ,\xi }+F\right) -W^{\ast }\right\vert ^{2}\varphi _{\lambda }\left( \mathbf{x%
  },t\right) d\mathbf{x}dt+ \\ 
  +\left( C_{1}/\lambda \right) \xi +C_{1}\delta ^{2}.%
  \end{array}%
  \right.   \label{7.100}
  \end{equation}
  \begin{equation}
  \left. 
  \begin{array}{c}
  e^{-2\lambda b^{2}}\dint\limits_{Q_{T}}\left[ \left( \beta _{n,\lambda
  ,\delta }\left( \mathbf{x}\right) -\beta ^{\ast }\left( \mathbf{x}\right)
  \right) ^{2}+\left( \gamma _{n,\lambda ,\delta }\left( \mathbf{x}\right)
  -\gamma ^{\ast }\left( \mathbf{x}\right) \right) ^{2}\right] \varphi
  _{\lambda }\left( \mathbf{x},t\right) d\mathbf{x}dt\leq  \\ 
  \leq \left( C_{1}/\lambda \right) ^{n}e^{-2\lambda
  b^{2}}\dint\limits_{Q_{T}}\left\vert \nabla \left( \widetilde{W}_{0,\min
  ,\lambda ,\xi }+F\right) -\nabla W^{\ast }\right\vert ^{2}\varphi _{\lambda
  }\left( \mathbf{x},t\right) d\mathbf{x}dt+ \\ 
  +\left( C_{1}/\lambda \right) ^{n}e^{-2\lambda
  b^{2}}\dint\limits_{Q_{T}}\left\vert \left( \widetilde{W}_{0,\min ,\lambda
  ,\xi }+F\right) -W^{\ast }\right\vert ^{2}\varphi _{\lambda }\left( \mathbf{x%
  },t\right) d\mathbf{x}dt+ \\ 
  +\left( C_{1}/\lambda \right) \xi +C_{1}\delta ^{2}.%
  \end{array}%
  \right.   \label{7.200}
  \end{equation}
\end{theorem}

\begin{proof} First, we consider the case $n=1.$ Using the analog of (\ref%
  {6.5}) for $I_{n,\lambda ,\xi },$ we obtain%
  \begin{equation}
  \left. 
  \begin{array}{c}
  I_{1,\lambda ,\xi }\left( \widetilde{W}^{\ast }\right) -I_{1,\lambda ,\xi
  }\left( \widetilde{W}_{n,\min ,\lambda ,\xi }\right) - \\ 
  -I_{1,\lambda ,\xi }^{\prime }\left( \widetilde{W}_{n,\min ,\lambda ,\xi
  }\right) \left( \widetilde{W}^{\ast }-\widetilde{W}_{n,\min ,\lambda ,\xi
  }\right) \geq  \\ 
  \geq C_{1}\lambda \dint\limits_{Q_{T}}\left[ \left( \nabla h_{1}\right)
  ^{2}+h_{1}^{2}\right] \varphi _{\lambda }\left( \mathbf{x},t\right) d\mathbf{%
  x}dt.%
  \end{array}%
  \right.   \label{7.11}
  \end{equation}%
  Since $-I_{1,\lambda ,\xi }\left( \widetilde{W}_{n,\min ,\lambda ,\xi
  }\right) \leq 0$ and also since by (\ref{6.6}) 
  \begin{equation*}
  -I_{1,\lambda ,\xi }^{\prime }\left( \widetilde{W}_{n,\min ,\lambda ,\xi
  }\right) \left( \widetilde{W}^{\ast }-\widetilde{W}_{n,\min ,\lambda ,\xi
  }\right) \leq 0,
  \end{equation*}%
  then (\ref{7.11}) implies%
  \begin{equation}
  e^{-2\lambda b^{2}}\dint\limits_{Q_{T}}\left[ \left( \nabla h_{1}\right)
  ^{2}+h_{1}^{2}\right] \varphi _{\lambda }\left( \mathbf{x},t\right) d\mathbf{%
  x}dt\leq \frac{C_{1}}{\lambda }I_{1,\lambda ,\xi }\left( \widetilde{W}^{\ast
  }\right) .  \label{7.12}
  \end{equation}
  
  We now estimate $I_{1,\lambda ,\xi }\left( \widetilde{W}^{\ast }\right) $
  from the above. By (\ref{6.110}) and (\ref{7.6})%
  \begin{equation}
  \left. 
  \begin{array}{c}
  I_{1,\lambda ,\xi }\left( \widetilde{W}^{\ast }\right) =J_{1,\lambda ,\xi
  }\left( \widetilde{W}^{\ast }+F\right) = \\ 
  =e^{-2\lambda b^{2}}\int\limits_{Q_{T}}\left[ L\left( \widetilde{W}^{\ast
  }+F\right) +Y\left( \widetilde{W}_{0,\min ,\lambda }+F,S\right) \right]
  ^{2}\varphi _{\lambda }d\mathbf{x}dt\mathbb{+} \\ 
  +\mathbb{\xi }\left\Vert \widetilde{W}^{\ast }+F\right\Vert
  _{H_{6}^{4}\left( Q_{T}\right) }^{2}.%
  \end{array}%
  \right.   \label{7.13}
  \end{equation}%
  Next consider the expression in the second line of (\ref{7.13}). We have:%
  \begin{equation*}
  \left. 
  \begin{array}{c}
  L\left( \widetilde{W}^{\ast }+F\right) +Y\left( \widetilde{W}_{0,\min
  ,\lambda }+F,S\right) =L\left( W^{\ast }\right) + \\ 
  +L\left( F-F^{\ast }\right) +Y\left( \widetilde{W}_{0,\min ,\lambda
  }+F,S\right) = \\ 
  =\left[ L\left( W^{\ast }\right) +Y\left( W^{\ast },S^{\ast }\right) \right]
  +L\left( F-F^{\ast }\right)  \\ 
  +\left[ Y\left( \widetilde{W}_{0,\min ,\lambda }+F,S\right) -Y\left( 
  \widetilde{W}^{\ast }+F^{\ast },S^{\ast }\right) \right] .%
  \end{array}%
  \right. 
  \end{equation*}%
  By (\ref{7.1}) $L\left( W^{\ast }\right) +Y\left( W^{\ast },S^{\ast }\right)
  =0.$ Hence, 
  \begin{equation}
  \left. 
  \begin{array}{c}
  L\left( \widetilde{W}^{\ast }+F\right) +Y\left( W_{0,\min ,\lambda
  },S\right) =L\left( F-F^{\ast }\right) + \\ 
  +\left[ Y\left( \widetilde{W}_{0,\min ,\lambda }+F,S\right) -Y\left( 
  \widetilde{W}^{\ast }+F^{\ast },S^{\ast }\right) \right] .%
  \end{array}%
  \right.   \label{7.14}
  \end{equation}%
  Next, by (\ref{7.3})%
  \begin{equation}
  \left\vert L\left( F-F^{\ast }\right) \right\vert \leq C_{1}\delta ,
  \label{7.15}
  \end{equation}%
  and by (\ref{5.5}), (\ref{7.3}) and (\ref{7.5}) 
  \begin{equation}
  \left. 
  \begin{array}{c}
  \left\vert Y\left( \widetilde{W}_{0,\min ,\lambda }+F,S\right) -Y\left( 
  \widetilde{W}^{\ast }+F^{\ast },S^{\ast }\right) \right\vert \leq  \\ 
  \leq C_{1}\left( \left\vert \widetilde{W}_{0,\min ,\lambda }-\widetilde{W}%
  ^{\ast }\right\vert \left( \mathbf{x},t\right) +\left\vert \nabla \widetilde{%
  W}_{0,\min ,\lambda }-\nabla \widetilde{W}^{\ast }\right\vert \left( \mathbf{%
  x},t\right) \right) + \\ 
  +C_{1}\left\vert \int\limits_{T/2}^{t}\left\vert \widetilde{W}_{0,\min
  ,\lambda }-\widetilde{W}^{\ast }\right\vert \left( \mathbf{x},\tau \right)
  d\tau \right\vert + \\ 
  +C_{1}\left\vert \int\limits_{T/2}^{t}\left\vert \nabla \widetilde{W}%
  _{0,\min ,\lambda }-\nabla \widetilde{W}^{\ast }\right\vert \left( \mathbf{x}%
  ,\tau \right) d\tau \right\vert +C_{1}\delta .%
  \end{array}%
  \right.   \label{7.16}
  \end{equation}%
  By Theorem 4.2, Cauchy-Schwarz inequality, (\ref{4.2}), (\ref{7.8}) and (\ref%
  {7.14})-(\ref{7.16}) 
  \begin{equation}
  \left. 
  \begin{array}{c}
  e^{-2\lambda b^{2}}\int\limits_{Q_{T}}\left[ L\left( \widetilde{W}^{\ast
  }+F\right) +Y\left( \widetilde{W}_{0,\min ,\lambda }+F,S\right) \right]
  ^{2}\varphi _{\lambda }d\mathbf{x}dt\leq  \\ 
  \leq C_{1}e^{-2\lambda b^{2}}\int\limits_{Q_{T}}\left[ \left\vert \nabla 
  \widetilde{W}_{0,\min ,\lambda }-\nabla \widetilde{W}^{\ast }\right\vert
  ^{2}+\left\vert \widetilde{W}_{0,\min ,\lambda }-\widetilde{W}^{\ast
  }\right\vert ^{2}\right] \varphi _{\lambda }d\mathbf{x}dt+ \\ 
  +C_{1}\delta ^{2}=C_{1}e^{-2\lambda b^{2}}\int\limits_{Q_{T}}\left(
  \left\vert \nabla h_{0}\right\vert ^{2}+\left\vert h_{0}\right\vert
  ^{2}\right) \varphi _{\lambda }\left( \mathbf{x},t\right) d\mathbf{x}%
  dt+C_{1}\delta ^{2}.%
  \end{array}%
  \right.   \label{7.17}
  \end{equation}%
  Finally 
  \begin{equation}
  \left. 
  \begin{array}{c}
  \mathbb{\xi }\left\Vert \widetilde{W}^{\ast }+F\right\Vert _{H_{6}^{4}\left(
  Q_{T}\right) }^{2}=\mathbb{\xi }\left\Vert W^{\ast }+\left( F-F^{\ast
  }\right) \right\Vert _{H_{6}^{4}\left( Q_{T}\right) }^{2}\leq  \\ 
  \leq 2\xi \left\Vert W^{\ast }\right\Vert _{H_{6}^{4}\left( Q_{T}\right)
  }^{2}+2\delta ^{2}\leq 2\xi M^{2}+2\delta ^{2}.%
  \end{array}%
  \right.   \label{7.18}
  \end{equation}%
  Hence, (\ref{7.13})-(\ref{7.18}) imply%
  \begin{equation*}
  I_{1,\lambda ,\xi }\left( \widetilde{W}^{\ast }\right) \leq e^{-2\lambda
  b^{2}}\dint\limits_{Q_{T}}\left( \left\vert \nabla h_{0}\right\vert
  ^{2}+\left\vert h_{0}\right\vert ^{2}\right) \varphi _{\lambda }\left( 
  \mathbf{x},t\right) d\mathbf{x}dt+C_{1}\left( \xi +\delta ^{2}\right) .
  \end{equation*}%
  Combining this with (\ref{7.12}), we obtain%
  \begin{equation*}
  \left. 
  \begin{array}{c}
  e^{-2\lambda b^{2}}\dint\limits_{Q_{T}}\left( \left\vert \nabla
  h_{1}\right\vert ^{2}+\left\vert h_{1}\right\vert ^{2}\right) \varphi
  _{\lambda }\left( \mathbf{x},t\right) d\mathbf{x}dt\leq  \\ 
  \leq \left( C_{1}/\lambda \right) e^{-2\lambda
  b^{2}}\dint\limits_{Q_{T}}\left( \left\vert \nabla h_{0}\right\vert
  ^{2}+\left\vert h_{0}\right\vert ^{2}\right) \varphi _{\lambda }\left( 
  \mathbf{x},t\right) d\mathbf{x}dt+\left( C_{1}/\lambda \right) \left( \xi
  +\delta ^{2}\right) ,%
  \end{array}%
  \right. 
  \end{equation*}%
  which is estimate (\ref{7.10}) at $n=1.$ Continuing these estimates via the
  mathematical induction, we obtain 
  \begin{equation}
  \left. 
  \begin{array}{c}
  e^{-2\lambda b^{2}}\dint\limits_{Q_{T}}\left( \left\vert \nabla
  h_{n}\right\vert ^{2}+\left\vert h_{n}\right\vert ^{2}\right) \varphi
  _{\lambda }\left( \mathbf{x},t\right) d\mathbf{x}dt\leq  \\ 
  \leq \left( C_{1}/\lambda \right) ^{n}e^{-2\lambda
  b^{2}}\dint\limits_{Q_{T}}\left( \left\vert \nabla h_{0}\right\vert
  ^{2}+\left\vert h_{0}\right\vert ^{2}\right) \varphi _{\lambda }\left( 
  \mathbf{x},t\right) d\mathbf{x}dt+ \\ 
  +\left( \dsum\limits_{k=1}^{n}\left( C_{1}/\lambda \right) ^{k}\right)
  \left( \xi +\delta ^{2}\right) .%
  \end{array}%
  \right.   \label{7.19}
  \end{equation}%
  Since 
  \begin{equation*}
  \dsum\limits_{k=1}^{\infty }\left( \frac{C_{1}}{\lambda }\right) ^{k}=\frac{%
  C_{1}}{\lambda }\cdot \frac{1}{1-\left( C_{1}/\lambda \right) }<2\frac{C_{1}%
  }{\lambda },
  \end{equation*}%
  then (\ref{7.19}) implies the first target estimate (\ref{7.10}).
  
  Using $h_{n}=\widetilde{W}_{n,\min ,\lambda ,\xi }-\widetilde{W}^{\ast }=%
  \left[ \left( \widetilde{W}_{n,\min ,\lambda ,\xi }+F\right) -W^{\ast }%
  \right] -\left( F-F^{\ast }\right) $ and (\ref{7.3}), we obtain%
  \begin{equation}
  \left. 
  \begin{array}{c}
  \left\vert \nabla h_{n}\right\vert ^{2}+\left\vert h_{n}\right\vert ^{2}\geq
  \left\vert \nabla \left( \widetilde{W}_{n,\min ,\lambda ,\xi }+F\right)
  -\nabla W^{\ast }\right\vert ^{2}+ \\ 
  +\left\vert \left( \widetilde{W}_{n,\min ,\lambda ,\xi }+F\right) -W^{\ast
  }\right\vert ^{2}-\delta ^{2}.%
  \end{array}%
  \right.   \label{7.20}
  \end{equation}%
  The second target estimate (\ref{7.100}) easily follows from (\ref{7.10})
  and (\ref{7.20}). Estimate (\ref{7.100}) being combined with (\ref{3.11}), (%
  \ref{3.4}) and (\ref{3.6}) leads to (\ref{7.200}).
\end{proof}

\begin{theorem}[convergence for noiseless data]\label{Theorem 7.2}{Let }$\lambda
  \geq \lambda _{1}${, where }$\lambda _{1}\geq 1${\ was found in
  Theorem 6.1. Let the noise level }$\delta =0${\ and conditions (\ref%
  {6.4}), (\ref{7.1})-(\ref{7.2}) as well as the notation in the second line
  of (\ref{7.8}) hold. Then the following convergence estimates\ are valid:}%
  \begin{equation}
  \left. 
  \begin{array}{c}
  e^{-2\lambda b^{2}}\dint\limits_{Q_{T}}\left( \left\vert \nabla
  s_{n}\right\vert ^{2}+\left\vert s_{n}\right\vert ^{2}\right) \varphi
  _{\lambda }\left( \mathbf{x},t\right) d\mathbf{x}dt\leq  \\ 
  \leq \left( C_{1}/\lambda \right) ^{n}e^{-2\lambda
  b^{2}}\dint\limits_{Q_{T}}\left( \left\vert \nabla s_{0}\right\vert
  ^{2}+\left\vert s_{0}\right\vert ^{2}\right) \varphi _{\lambda }\left( 
  \mathbf{x},t\right) d\mathbf{x}dt+\left( C_{1}/\lambda \right) \xi ,\text{ }%
  n=1,2,...%
  \end{array}%
  \right.   \label{7.101}
  \end{equation}%
  \begin{equation}
  \left. 
  \begin{array}{c}
  e^{-2\lambda b^{2}}\dint\limits_{Q_{T}}\left[ \left( \beta _{n,\lambda
  }\left( \mathbf{x}\right) -\beta ^{\ast }\left( \mathbf{x}\right) \right)
  ^{2}+\left( \gamma _{n,\lambda }\left( \mathbf{x}\right) -\gamma ^{\ast
  }\left( \mathbf{x}\right) \right) ^{2}\right] \varphi _{\lambda }\left( 
  \mathbf{x},t\right) d\mathbf{x}dt\leq  \\ 
  \leq \left( C_{1}/\lambda \right) ^{n}e^{-2\lambda
  b^{2}}\dint\limits_{Q_{T}}\left( \left\vert \nabla s_{0}\right\vert
  ^{2}+\left\vert s_{0}\right\vert ^{2}\right) \varphi _{\lambda }\left( 
  \mathbf{x},t\right) d\mathbf{x}dt+\left( C_{1}/\lambda \right) \xi ,\text{ }%
  n=1,2,...%
  \end{array}%
  \right.   \label{7.300}
  \end{equation}
\end{theorem}

We omit the proof of this theorem since it is completely similar with the
proof of Theorem 7.1. It follows from (\ref{6.4}) and\emph{\ }(\ref{7.100})
that the sequence $\left\{ \widetilde{W}_{n,\min ,\lambda ,\xi }+F\right\}
_{n=0}^{\infty }$ converges to the exact solution $W^{\ast }$ in the case of
noisy data, as long as the level of noise in the data $\delta \rightarrow 0$%
. And (\ref{7.101}) implies that the sequence $\left\{ W_{n,\min ,\lambda
,\xi }\right\} _{n=0}^{\infty }$ also converges to the exact solution $%
W^{\ast }$ in the case of noiseless data. Estimates (\ref{7.200}) and (\ref%
{7.300}) imply that the same statements are true for the target
reconstructed coefficients. In all cases convergence is in terms of
integrals with the CWF $\varphi _{\lambda }\left( \mathbf{x},t\right) $ in
them. Finally, Theorems 6.2, 6.3, 7.1 and 7.2 imply that we have constructed
in section 5 a globally convergent numerical method for our CIP.

We now provide more explicit convergence estimates for the unknown
coefficients. Recall that the domain $Q_{\alpha T}$ is the one defined in (%
\ref{2.1}) for the number $\alpha $ in (\ref{1}). We also recall that by the
regularization theory, the regularization parameter should depend on the
noise level \cite{T}.

\begin{theorem}\label{Theorem 7.3}
  {Let conditions of Theorem 7.1 hold, so as (\ref%
  {1}). Let }%
  \begin{equation}
  T^{2}>\frac{8b^{2}}{1-2\alpha ^{2}}.  \label{7.301}
  \end{equation}%
  {Define two numbers }$m,s${\ as}%
  \begin{equation}
  m=\alpha ^{2}\frac{T^{2}}{2}+2b^{2},\text{ }s=\frac{T^{2}}{4}\left[ \left(
  1-2\alpha ^{2}\right) -\frac{8b^{2}}{T^{2}}\right] .  \label{7.302}
  \end{equation}%
  {Suppose that the number }$\delta _{0}\in \left( 0,1\right) ${\ is
  so small that} 
  \begin{equation}
  \ln \left( \delta _{0}^{-1/m}\right) \geq \lambda _{2}.  \label{7.303}
  \end{equation}%
  {For every }$\delta \in \left( 0,\delta _{0}\right) ,${\ let }$%
  \lambda =\lambda \left( \delta \right) =\ln \left( \delta ^{-1/m}\right) .$%
  {\ Following (\ref{6.4}), choose the regularization parameter }$\xi $ 
  {as}%
  \begin{equation}
  \xi \left( \delta \right) =2\exp \left( -\lambda \left( \delta \right) \frac{%
  T^{2}}{4}\right) .  \label{7.304}
  \end{equation}%
  {Then the following convergence estimate holds:}%
  \begin{equation}
  \left. 
  \begin{array}{c}
  \left\Vert \beta _{n,\lambda \left( \delta \right) ,\delta }-\beta ^{\ast
  }\right\Vert _{L_{2}\left( \Omega \right) }^{2}+\left\Vert \gamma
  _{n,\lambda \left( \delta \right) ,\delta }-\gamma ^{\ast }\right\Vert
  _{L_{2}\left( \Omega \right) }^{2}\leq  \\ 
  \leq \alpha ^{-1}\left( C_{1}/\lambda \left( \delta \right) \right) ^{n}\exp
  \left( \left[ \lambda \left( \delta \right) \left( \alpha
  ^{2}T^{2}/2+2b^{2}\right) \right] \right) \times  \\ 
  \times \left[ \left\Vert \nabla \left( \widetilde{W}_{0,\min ,\lambda ,\xi
  }+F\right) -\nabla W^{\ast }\right\Vert _{L_{2}^{2}\left( Q_{T}\right)
  }^{2}+\left\Vert \left( \widetilde{W}_{0,\min ,\lambda ,\xi }+F\right)
  -W^{\ast }\right\Vert _{L^{2}\left( Q_{T}\right) }^{2}\right] + \\ 
  +C_{1}\alpha ^{-1}\delta ^{\rho },\text{ }%
  \end{array}%
  \right.   \label{7.305}
  \end{equation}%
  {for any} $\alpha \in \left( 0,1/\sqrt{2}\right) ,$ {where }$\rho
  =\max \left( 1,s/m\right) .$
\end{theorem}

\begin{proof}
  Since $Q_{\alpha T}\subset Q_{T},$ then, using (\ref{4.2}),
  we obtain%
  \begin{equation*}
  \left. 
  \begin{array}{c}
  e^{-2\lambda b^{2}}\dint\limits_{Q_{T}}\left[ \left( \beta _{n,\lambda
  ,\delta }\left( \mathbf{x}\right) -\beta ^{\ast }\left( \mathbf{x}\right)
  \right) ^{2}+\left( \gamma _{n,\lambda ,\delta }\left( \mathbf{x}\right)
  -\gamma ^{\ast }\left( \mathbf{x}\right) \right) ^{2}\right] \varphi
  _{\lambda }\left( \mathbf{x},t\right) d\mathbf{x}dt\geq  \\ 
  \geq e^{-2\lambda b^{2}}\dint\limits_{Q_{\alpha T}}\left[ \left( \beta
  _{n,\lambda ,\delta }\left( \mathbf{x}\right) -\beta ^{\ast }\left( \mathbf{x%
  }\right) \right) ^{2}+\left( \gamma _{n,\lambda ,\delta }\left( \mathbf{x}%
  \right) -\gamma ^{\ast }\left( \mathbf{x}\right) \right) ^{2}\right] \varphi
  _{\lambda }\left( \mathbf{x},t\right) d\mathbf{x}dt\geq  \\ 
  \geq \alpha C_{1}e^{-2\lambda b^{2}}\exp \left( -\lambda \alpha
  ^{2}T^{2}/2\right) \left( \left\Vert \beta _{n,\lambda \left( \delta \right)
  ,\delta }-\beta ^{\ast }\right\Vert _{L^{2}\left( \Omega \right)
  }^{2}+\left\Vert \gamma _{n,\lambda \left( \delta \right) ,\delta }-\gamma
  ^{\ast }\right\Vert _{L^{2}\left( \Omega \right) }^{2}\right) .%
  \end{array}%
  \right. 
  \end{equation*}%
  Combining this with (\ref{7.200}), (\ref{7.301})-(\ref{7.304}), we obtain%
  \begin{equation}
  \left. 
  \begin{array}{c}
  \left\Vert \beta _{n,\lambda \left( \delta \right) ,\delta }-\beta ^{\ast
  }\right\Vert _{L^{2}\left( \Omega \right) }^{2}+\left\Vert \gamma
  _{n,\lambda \left( \delta \right) ,\delta }-\gamma ^{\ast }\right\Vert
  _{L^{2}\left( \Omega \right) }^{2}\leq \alpha ^{-1}C_{1}e^{-\lambda s}+ \\ 
  +\alpha ^{-1}C_{1}e^{2\lambda b^{2}}\exp \left( -\lambda \alpha
  ^{2}T^{2}/2\right) \delta ^{2}.%
  \end{array}%
  \right.   \label{7.306}
  \end{equation}%
  For the above choice $\lambda =\lambda \left( \delta \right) $ terms in the
  right hand side of (\ref{7.306}) are:
  \begin{equation*}
  e^{2\lambda b^{2}}\exp \left( -\lambda \alpha ^{2}T^{2}/2\right) \delta
  ^{2}=\delta ,\text{ }e^{-\lambda s}=\delta ^{s/m}.
  \end{equation*}
  This completes the proof.
\end{proof}

We now formulate an analogous theorem for the case when conditions of %
\Cref{Theorem 7.2} are valid.

\begin{theorem} \label{Theorem 7.4}
  {\ Let conditions of Theorem 7.2 hold. Assume that
  inequality (\ref{7.301}) is valid and let }$s${\ be the number defined
  in (\ref{7.302}). Let the regularization parameter }$\xi =2\exp \left(
  -\lambda T^{2}/4\right) .${\ Then the following convergence estimate
  holds for any }$\alpha \in \left( 0,1/\sqrt{2}\right) ${\ and for all }$%
  \lambda \geq \lambda _{1}:$%
  \begin{equation*}
  \left. 
  \begin{array}{c}
  \left\Vert \beta _{n,\lambda }-\beta ^{\ast }\right\Vert _{L^{2}\left(
  \Omega \right) }^{2}+\left\Vert \gamma _{n,\lambda }-\gamma ^{\ast
  }\right\Vert _{L^{2}\left( \Omega \right) }^{2}\leq  \\ 
  \leq \alpha ^{-1}\left( C_{1}/\lambda \right) ^{n}\exp \left( \left[ \lambda
  \left( \alpha ^{2}T^{2}/2+2b^{2}\right) \right] \right) \times  \\ 
  \times \left[ \left\Vert \nabla W_{0,\min ,\lambda ,\xi }-\nabla W^{\ast
  }\right\Vert _{L_{2}^{2}\left( Q_{T}\right) }^{2}+\left\Vert W_{0,\min
  ,\lambda ,\xi }-W^{\ast }\right\Vert _{L^{2}\left( Q_{T}\right) }^{2}\right]
  +C_{1}e^{-\lambda s}.%
  \end{array}%
  \right. 
  \end{equation*}
\end{theorem}

We omit the proof of this theorem since it is similar with the proof of %
\Cref{Theorem 7.3}.

\section{Numerical Studies}

\label{sec:8}

In this section, we discuss some tests, which demonstrate the numerical
performance of the method. We note first that even though our Theorems
7.1-7.4 require sufficiently large values of $\lambda >1,$ we establish
numerically in subsection 8.1 that $\lambda =5$ can be chosen as an optimal
value of this parameter. This coincides with the observation of many past
publications on the convexification method: we refer to e.g. \cite%
{Bak,KL,Epid,SAR,KTR,Le}, in which optimal values of $\lambda $ were $%
\lambda \in \left[ 1,5\right] .$ This is similar with many cases of the
asymptotic theories. Indeed, such a theory sometimes states that if a
certain parameter $X$ is sufficiently large, then a certain formula $Y$ has
a good accuracy. However, in a computational practice only results of
numerical experiments can establish what exactly \textquotedblleft
sufficiently large $X$" means in computations.

To demonstrate the fact that our method can image rather complicated shapes
of inclusions, we choose letters--like shapes of inclusions. More precisely,
our shapes mimic letters $A,M,\Omega $ and $B.$ Indeed, these shapes are
non-convex and have voids. We are not concerned with the accurate imaging of
the fine structures of edges. Instead, we want to compute rather accurately
just shapes of inclusions and values of the unknown coefficients $\beta
\left( \mathbf{x}\right) $ and $\gamma \left( \mathbf{x}\right) .$

\subsection{Numerical setup}

\label{sec:8.1}

In our numerical simulations, the domain $\Omega =\{\mathbf{x}%
=(x,y):1<x<2,|y|<0.5\}$ and $T=1$ is the end time. \ The larger domain $G,$
where the forward problem (\ref{2.2})- (\ref{2.6}) is computed to generate
data for the inverse problem, is $G=\{\mathbf{x}=(x,y):(x-1.5)^{2}+y^{2}<1%
\}. $ The viscosity term is $d=0.1$, the velocities are $q_{s}=q_{I}=q_{R}=%
\left( 0.2,0.2\right) ^{T}.$ The Neumann boundary conditions (\ref{2.5}) and
the initial conditions (\ref{2.6}) are chosen as: 
\begin{equation*}
\left. 
\begin{array}{c}
g_{1}(\mathbf{x},t)=g_{2}(\mathbf{x},t)=g_{3}(\mathbf{x},t)=0, \\ 
\rho _{S}^{0}(\mathbf{x})=0.6,\quad \rho _{I}^{0}(\mathbf{x})=0.8,\quad \rho
_{R}^{0}(\mathbf{x})=0.%
\end{array}%
\right.
\end{equation*}

To solve the forward problem, we use the PDE Toolbox in MATLAB to implement
a Finite Element Method with the maximal mesh edge length 0.05. We compute
the data $r_{j}$, $f_{j}$, $j=1,2,3$ in (\ref{2.10}), (\ref{2.11}) for
eleven (11) equally spaced temporal points on $[0,T]$.

In all our numerical experiments $\beta \left( \mathbf{x}\right) =\gamma
\left( \mathbf{x}\right) =0.1$ in the data simulation process in the part of
the domain $G$, which is outside of our letters-like shapes. However, when
computing the inverse problems, we assume that functions $\beta \left( 
\mathbf{x}\right) $ are $\gamma \left( \mathbf{x}\right) $ unknown in the
entire domain $\Omega ,$ see (\ref{2.01}).

To solve the inverse problem, we write the differential operators of the
objective functional $J_{n,\lambda ,\xi }$ in (\ref{5.8}) in the form of
finite differences and minimize the so discretized functional with respect
to the values of the vector function $W$ at grid points. More precisely, we
discretize $\Omega $ into a $33\times 33$ meshgrid of equal edge length and $%
[0,T]$ into 11 equally spaced points. The data generated by the forward
solver are interpolated correspondingly to fit this mesh. We then add noise
to the data using different noise levels. For $j=1,2,3$, let $\mathbf{P}_{j}$%
, $\mathbf{R}_{j}$, $\mathbf{F}_{j}$ be the discretized matrices of $p_{j}(%
\mathbf{x})$, $r_{j}(\mathbf{x},t)$, $f_{j}(y,t)$, respectively. These
functions are defined in (\ref{2.9})-(\ref{2.11}). We use the following
noise model: given a matrix $\mathbf{A}$, the noisy version of $\mathbf{A}$
is 
\begin{equation*}
\mathbf{A}^{\text{noise}}=\mathbf{A}+\delta \Vert \mathbf{A}\Vert _{\infty }%
\text{Rand}(\mathbf{A}),
\end{equation*}%
where $\text{Rand}(\mathbf{A})$ is a matrix of the same size as $\mathbf{A}$
whose entries are uniformly generated random numbers on $(-1,1)$, $\Vert 
\mathbf{A}\Vert _{\infty }:=\max_{ij}\left\vert \mathbf{A}_{ij}\right\vert $
is the infinity norm of $\mathbf{A}$, and $\delta \in \lbrack 0,1)$ is the
noise level. Since we need to differentiate the data $r_{j}(\mathbf{x},t)$, $%
f_{j}(y,t)$ twice with respect to $t$, we first use cubic smoothing splines
to approximate the noisy data, then differentiate these splines with respect
to $t$. We proceed similarly to compute the required derivatives of $p_{j}(%
\mathbf{x})$. The regularization parameter\ $\xi =10^{-2}$ is chosen by
trial and error. For simplicity, we implement the $H_{6}^{2}(Q_{T})-$norm
for the regularization term instead of the $H_{6}^{4}(Q_{T})-$norm as in the
theory. We observe numerically that the $H_{6}^{2}(Q_{T})$ norm yields
sufficiently good results.

To minimize the discretized functional $J_{n,\lambda ,\xi }$, we use the 
\textbf{lsqlin} function of MATLAB. This is a function in the Optimization
Toolbox designed to solve linear problems using the least squares methods.
We stop the iterative process when the difference between the solutions of
the current step and the previous step is less than $10^{-5}$. This stopping
criterion is reached within $5$ iterations in all our numerical experiments.

We select the parameter of the Carleman Weight Function $\lambda =5$. See
Figure \ref{fig:lambda} for an experiment for different values of $\lambda $%
. In this experiment, we run our method to reconstruct $\gamma $, and $%
\lambda =5$ gives the best result in terms of both shape and value of $%
\gamma $. It is worth noting that the method fails when $\lambda =0$, i.e.
in the case when the CWF is absent in the objective functional $J_{n,\lambda
,\xi }.$ This further justifies the above theory since analogs of all our
theorems are invalid at $\lambda =0.$ We point out that, once chosen, the
value $\lambda =5$ is used in all other tests.

\begin{figure}[ht!]
\centering
\subfloat[True
$\gamma$]{\includegraphics[width=0.3\linewidth]{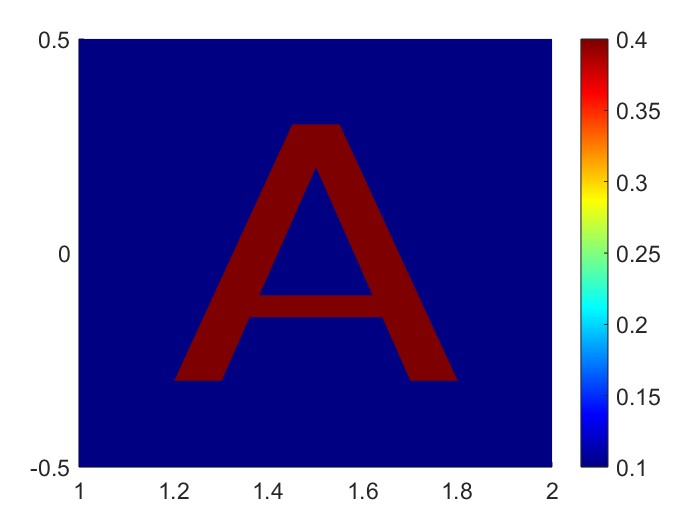}} 
\subfloat[$\lambda =
0$]{\includegraphics[width=0.3\linewidth]{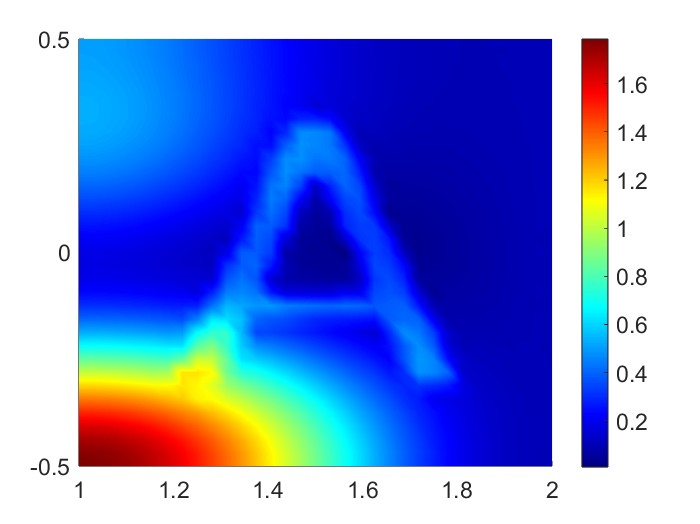}} 
\subfloat[$\lambda =
3$]{\includegraphics[width=0.3\linewidth]{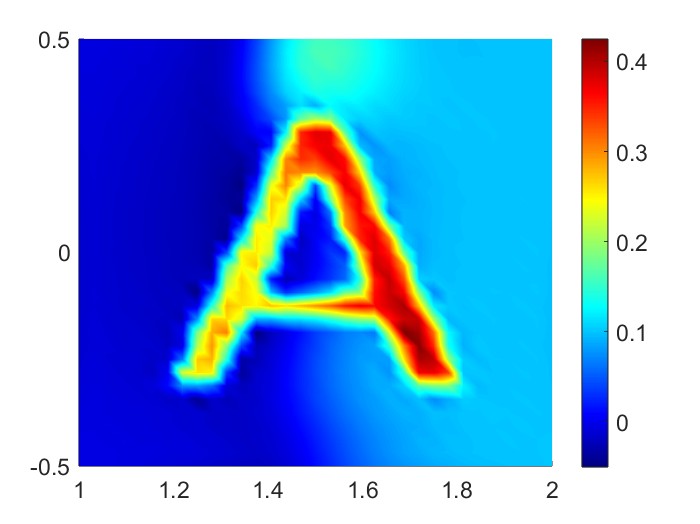}} \\
\subfloat[$\lambda =
5$]{\includegraphics[width=0.3\linewidth]{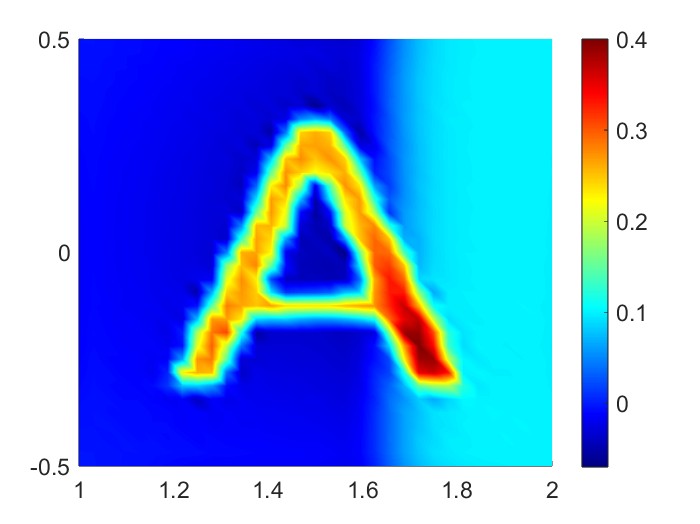}} 
\subfloat[$\lambda =
7$]{\includegraphics[width=0.3\linewidth]{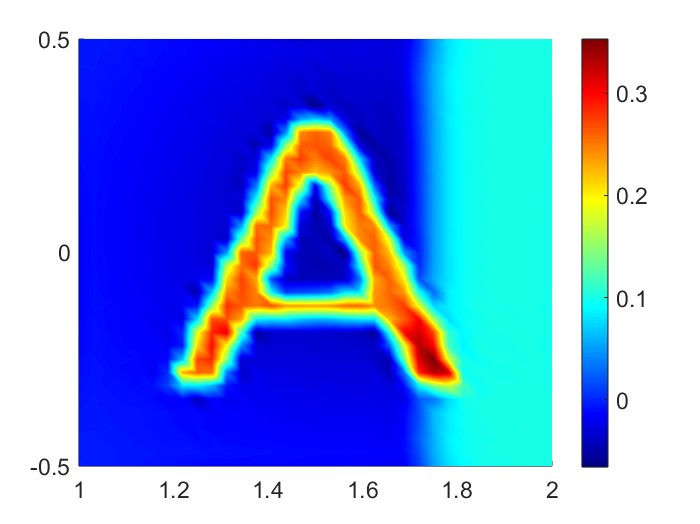}} 
\subfloat[$\lambda =
10$]{\includegraphics[width=0.3\linewidth]{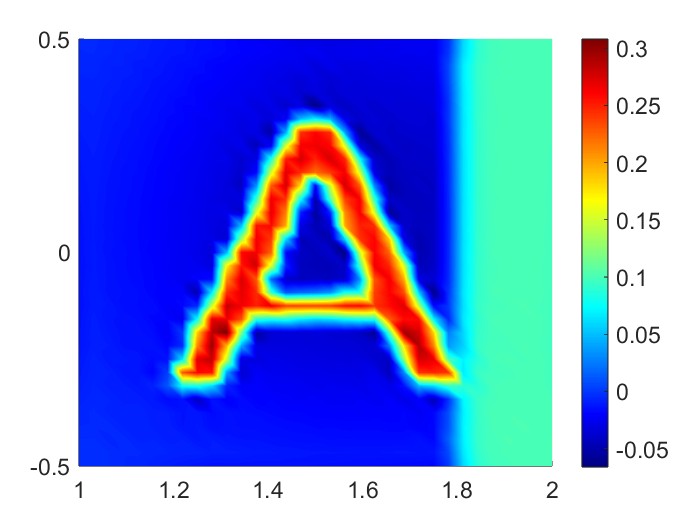}}
\caption{Experiment with different values of $\protect\lambda$. We choose $%
\protect\lambda =5$ as an optimal value of $\protect\lambda .$ This value is
used in all other tests.}
\label{fig:lambda}
\end{figure}

\subsection{Testing for different noise levels}

\label{sec:8.2}

In this test, we consider $\gamma $ and $\beta $ be as follows, 
\begin{equation*}
\gamma (\mathbf{x})=\left\{ 
\begin{array}{ll}
0.4 & \text{\textbf{$x$} inside of `}A\text{'} \\ 
0.1 & \text{\textbf{$x$} outside of `}A\text{'}%
\end{array}%
\right. ,\quad \beta (\mathbf{x})=\left\{ 
\begin{array}{ll}
0.6 & \text{\textbf{$x$} inside `}M\text{'} \\ 
0.1 & \text{\textbf{$x$} outside `}M\text{'}%
\end{array}%
\right. .
\end{equation*}%
Figures \ref{fig:gamma1} and \ref{fig:beta1} show the reconstructions of $%
\gamma $ and $\beta $ for at different noise levels: $\delta =0$, $\delta
=2\%$, and $\delta =5\%$. We obtain good reconstructions of both the shape
and the values of the unknown coefficients for the letters at lower noise
levels. When noise is $5\%$, the reconstruction of the letter `$M$' is
distorted, but we can still see the general shape and reasonable value.

\begin{figure}[ht!]
\centering
\subfloat[True
$\gamma$]{\includegraphics[width=0.3\linewidth]{gamma_true.jpg}} \\
\subfloat[$0\%$
noise]{\includegraphics[width=0.3\linewidth]{gamma_rec_noise0.jpg}} 
\subfloat[$2\%$
noise]{\includegraphics[width=0.3\linewidth]{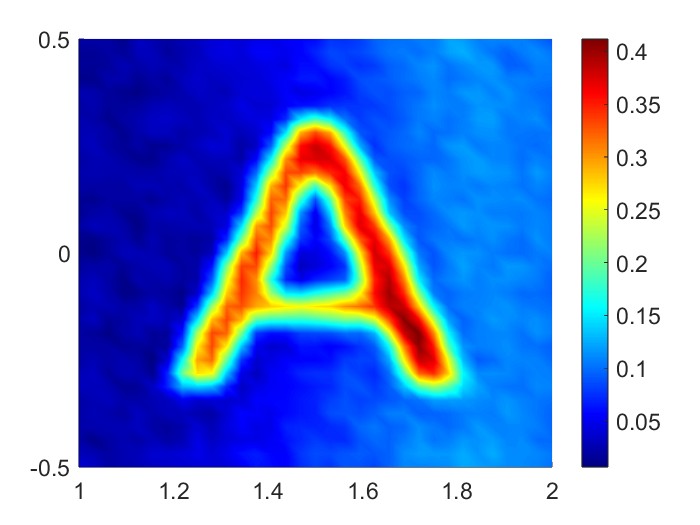}} 
\subfloat[$5\%$
noise]{\includegraphics[width=0.3\linewidth]{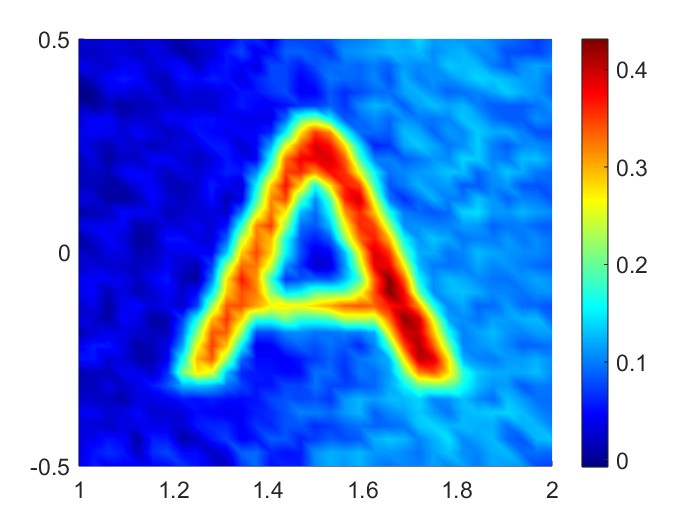}}
\caption{Reconstruction of $\protect\gamma$ (letter ``A'') at different
noise levels.}
\label{fig:gamma1}
\end{figure}

\begin{figure}[ht!]
\centering
\subfloat[True $\beta$]{\includegraphics[width=0.3\linewidth]{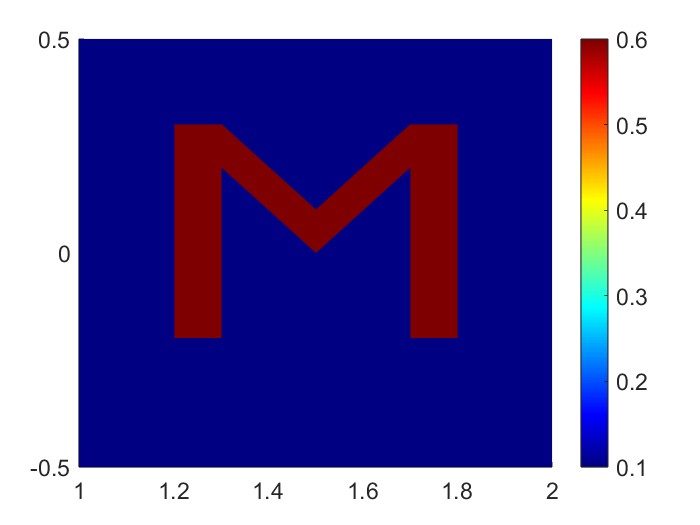}}
\\
\subfloat[$0\%$
noise]{\includegraphics[width=0.3\linewidth]{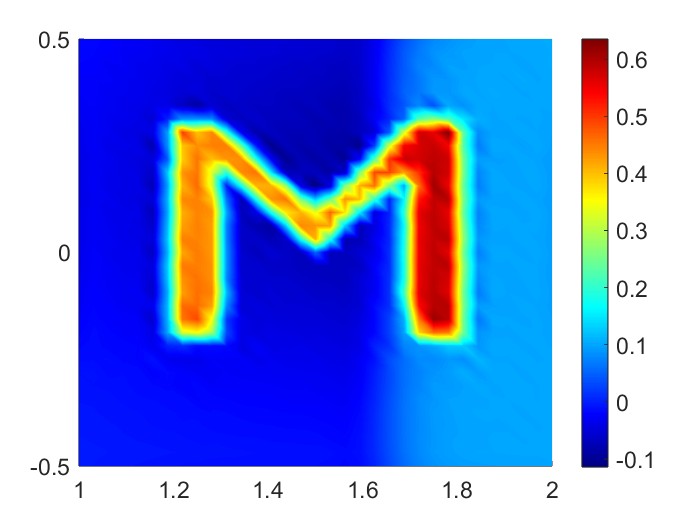}} 
\subfloat[$2\%$
noise]{\includegraphics[width=0.3\linewidth]{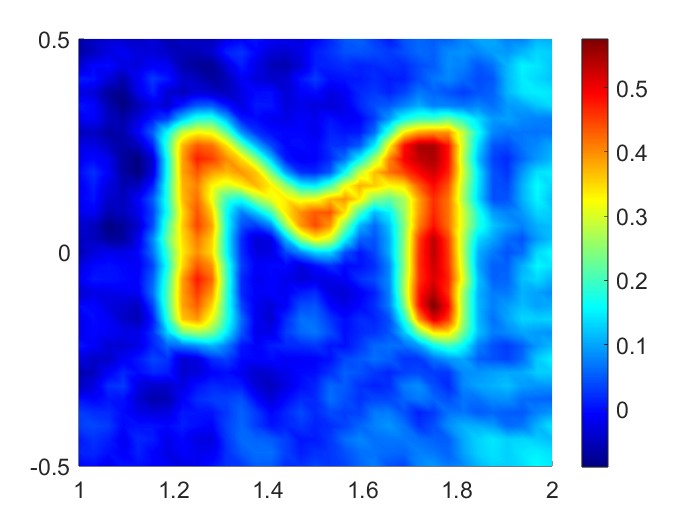}} 
\subfloat[$5\%$
noise]{\includegraphics[width=0.3\linewidth]{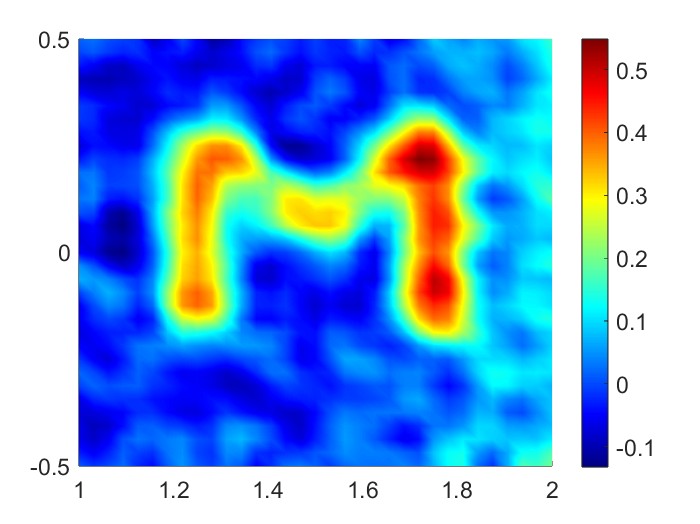}}
\caption{Reconstruction of $\protect\beta$ (letter ``M'') at different noise
levels.}
\label{fig:beta1}
\end{figure}

\subsection{Testing of different unknown coefficients}

\label{sec:8.3}

In this test, our method reconstructs different functions $\gamma (\mathbf{x}%
)$ and $\beta (\mathbf{x})$. The noise level for this test is $\delta =2\%$.
We consider two cases: 
\begin{equation*}
\gamma (\mathbf{x})=\left\{ 
\begin{array}{ll}
0.4 & \text{\textbf{$x$} inside `$\Omega $'} \\ 
0.1 & \text{\textbf{$x$} outside `$\Omega $'}%
\end{array}%
\right. ,\quad \beta (\mathbf{x})=\left\{ 
\begin{array}{ll}
0.6 & \text{\textbf{$x$} inside `}B\text{'} \\ 
0.1 & \text{\textbf{$x$} outside `}B\text{'}%
\end{array}%
\right.
\end{equation*}%
and 
\begin{equation*}
\gamma (\mathbf{x})=\left\{ 
\begin{array}{ll}
0.8 & \text{\textbf{$x$} inside `$\Omega $'} \\ 
0.1 & \text{\textbf{$x$} outside `$\Omega $'}%
\end{array}%
\right. ,\quad \beta (\mathbf{x})=\left\{ 
\begin{array}{ll}
1 & \text{\textbf{$x$} inside `}B\text{'} \\ 
0.1 & \text{\textbf{$x$} outside `}B\text{'}%
\end{array}%
\right. ,
\end{equation*}%
see Figures \ref{fig:case21} and \ref{fig:case22}, respectively. Even with
different shapes and values of $\gamma $ and $\beta $, our method still
performs quite well.

\begin{figure}[ht!]
\centering
\subfloat[True
$\gamma$]{\includegraphics[width=0.3\linewidth]{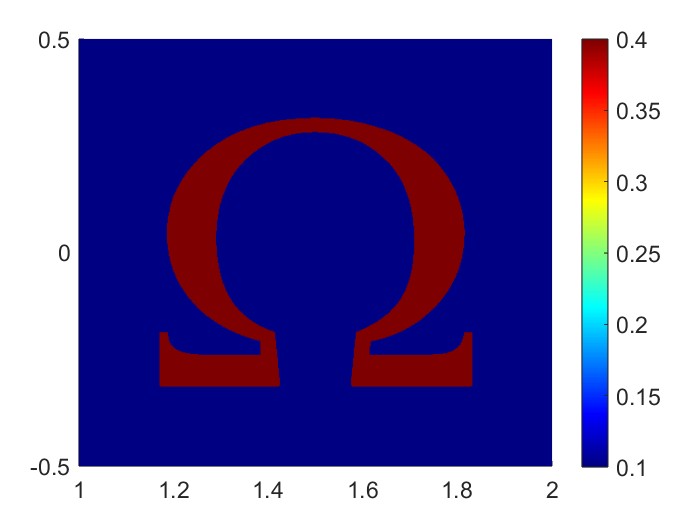}} 
\subfloat[Reconstructed
$\gamma$]{\includegraphics[width=0.3\linewidth]{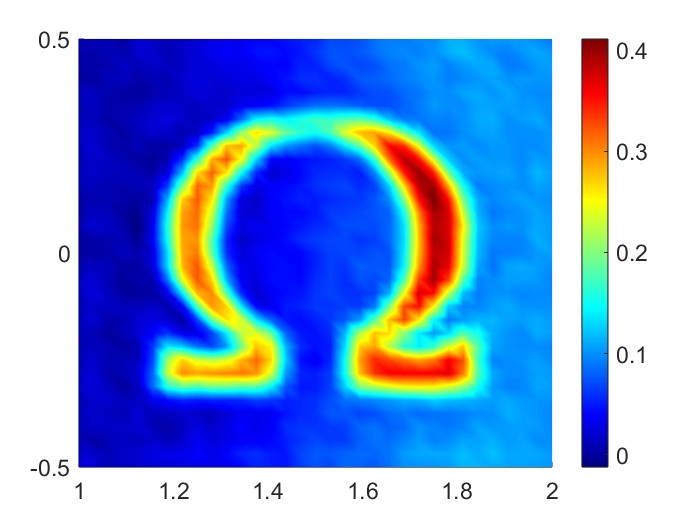}} 
\\
\subfloat[True
$\beta$]{\includegraphics[width=0.3\linewidth]{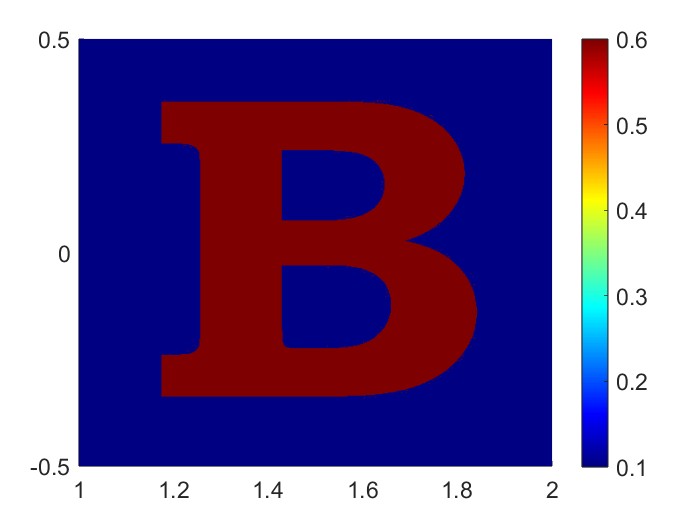}} 
\subfloat[Reconstructed
$\beta$]{\includegraphics[width=0.3\linewidth]{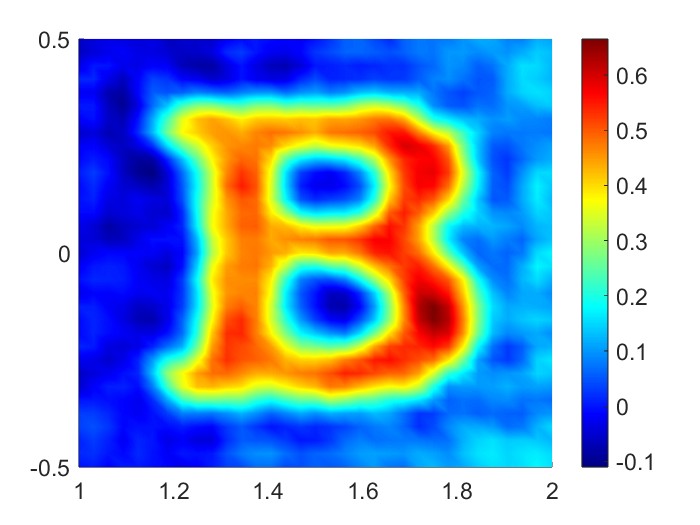}}
\caption{Reconstruction of $\protect\gamma$ (letter ``$\Omega$'') and $%
\protect\beta$ (letter ``B'') with values 0.4 and 0.6 respectively inside
these shapes and 0.1 outside at 2\% of the noise.}
\label{fig:case21}
\end{figure}

\begin{figure}[th]
\centering
\subfloat[True
$\gamma$]{\includegraphics[width=0.3\linewidth]{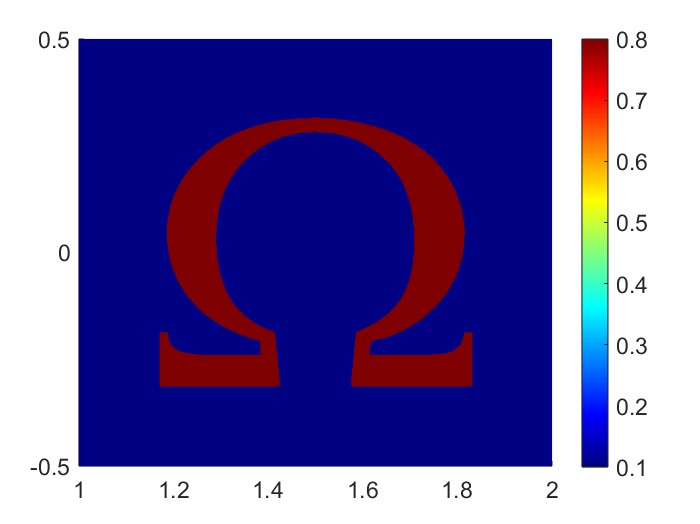}} 
\subfloat[Reconstructed
$\gamma$]{\includegraphics[width=0.3\linewidth]{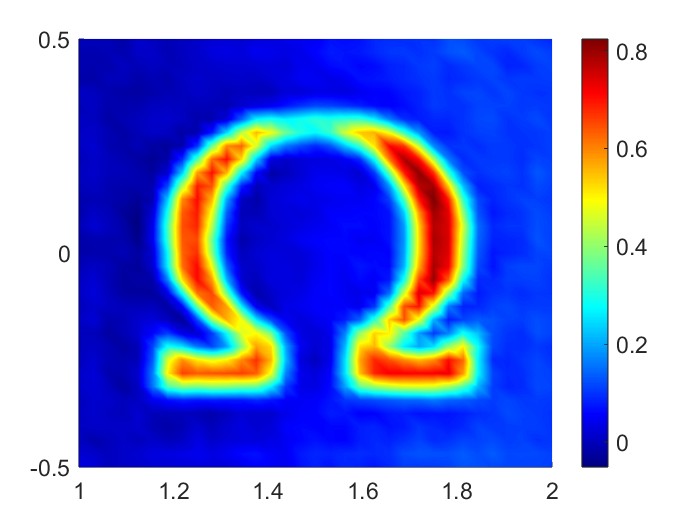}} 
\\
\subfloat[True
$\beta$]{\includegraphics[width=0.3\linewidth]{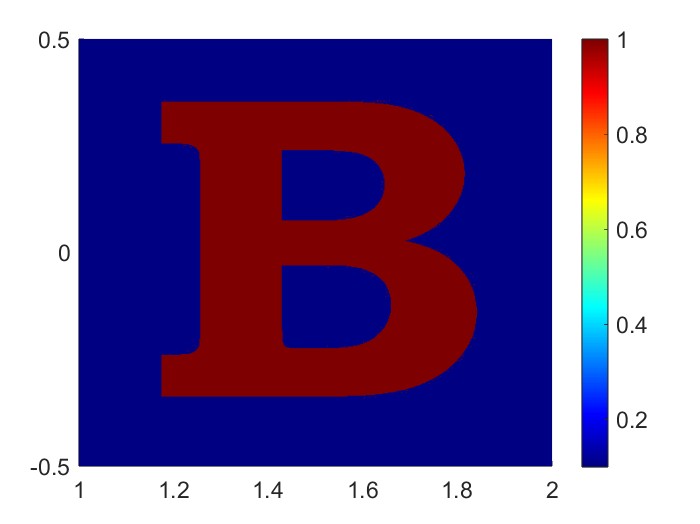}} 
\subfloat[Reconstructed
$\beta$]{\includegraphics[width=0.3\linewidth]{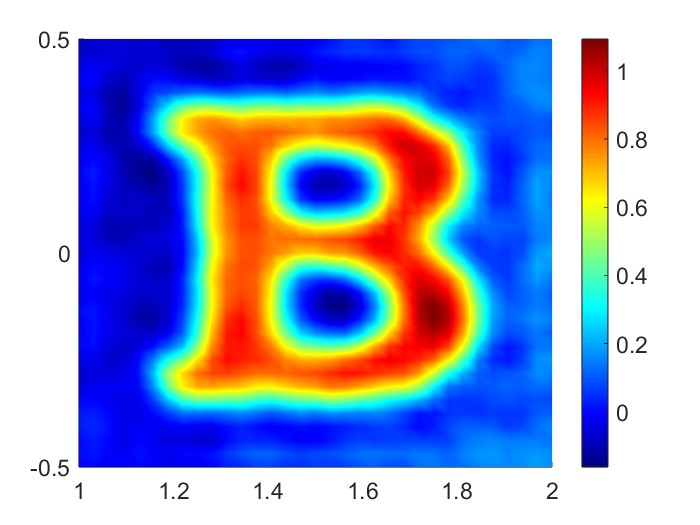}}
\caption{Reconstruction of $\protect\gamma $ (letter \textquotedblleft $%
\Omega $\textquotedblright ) and $\protect\beta $ (letter \textquotedblleft
B\textquotedblright ) with values 0.8 and 1 respectively inside these shapes
and 0.1 outside at 2\% of the noise.}
\label{fig:case22}
\end{figure}

\section*{Acknowledgments}

The work of M.~V.~Klibanov was partially supported by National Science
Foundation grant DMS 2436227.

\bibliographystyle{siamplain}
\bibliography{references}

\begin{thebibliography}{10}

\bibitem{B1}
{\sc M.~Asadzadeh and L.~Beilina}, {\em Stability and convergence analysis of a domain decomposition fe/fd method for maxwell's equations in the time domain}, Algorithms, 15 (2022), p.~337, \url{https://doi.org/10.1137/140951758}.

\bibitem{B2}
{\sc M.~Asadzadeh and L.~Beilina}, {\em A stabilized p1 domain decomposition finite element method for time harmonic maxwell's equations}, Math. Comput. Simul, 204 (2023), pp.~556--574.

\bibitem{Bak}
{\sc A.~B. Bakushinskii, M.~V. Klibanov, and N.~A. Koshev}, {\em Carleman weight functions for a globally convergent numerical method for ill-posed cauchy problems for some quasilinear pdes}, Nonlinear Analysis: Real World Applications, 34 (2017), pp.~201--224.

\bibitem{B3}
{\sc L.~Beilina, M.~G. Aram, and E.~M. Karchevskii}, {\em An adaptive finite element method for solving 3d electromagnetic volume integral equation with applications in microwave thermometry}, Journal of Computational Physics, 459 (2022), p.~111122.

\bibitem{Chavent}
{\sc G.~Chavent}, {\em Nonlinear Least Squares for Inverse Problems: Theoretical Foundations and Step-by-Step Guide for Applications}, Springer Science \& Business Media, New York, 2009.

\bibitem{Gonch1}
{\sc A.~V. Goncharsky and S.~Y. Romanov}, {\em Iterative methods for solving coefficient inverse problems of wave tomography in models with attenuation}, Inverse Probl., 33 (2017), p.~025003.

\bibitem{Gonch2}
{\sc A.~V. Goncharsky and S.~Y. Romanov}, {\em A method of solving the coefficient inverse problems of wave tomography}, Nonlinear Analysis: Real World Applications, 77 (2019), pp.~967--980.

\bibitem{Kermack}
{\sc W.~O. Kermack and A.~G. McKendrick}, {\em A contribution to the mathematical theory of epidemics}, Proc. Roy. Soc. London Ser. A, 115 (1927), pp.~700--721.

\bibitem{KL}
{\sc M.~Klibanov and J.~Li}, {\em Inverse Problems and Carleman Estimates: Global Uniqueness, Global Convergence and Experimental Data}, De Gruyter, 2021.

\bibitem{Epid}
{\sc M.~Klibanov, J.~Li, and Z.~Yang}, {\em Spatiotemporal monitoring of epidemics via solution of a coefficient inverse problem}, arXiv: 2401.02070,  (2024).

\bibitem{KI}
{\sc M.~V. Klibanov and O.~V. Ioussoupova}, {\em Uniform strict convexity of a cost functional for three-dimensional inverse scattering problem}, SIAM J. Math. Anal., 26 (1995), pp.~147--179.

\bibitem{SAR}
{\sc M.~V. Klibanov, V.~A. Khoa, A.~V. Smirnov, L.~H. Nguyen, G.~W. Bidney, L.~Nguyen, A.~Sullivan, and V.~N. Astratov}, {\em Convexification inversion method for nonlinear {SAR} imaging with experimentally collected data}, J. Appl. Ind. Math., 15 (2021), pp.~413--436.

\bibitem{KTR}
{\sc M.~V. Klibanov, J.~Li, L.~H. Nguyen, and Z.~Yang}, {\em Convexification numerical method for a coefficient inverse problem for the radiative transport equation}, SIAM J. Imag. Sci., 16 (2023), pp.~35--63.

\bibitem{KT}
{\sc M.~V. Klibanov and A.~A. Timonov}, {\em Carleman Estimates and Coefficient Inverse Problems and Numerical Applications}, VSP, Utrecht, 2004.

\bibitem{Lad}
{\sc O.~A. Ladyzhenskaya, V.~A. Solonnikov, and N.~N. Uralceva}, {\em Linear and Quasilinear Equations of Parabolic Type}, AMS, Providence, R.I., 1968.

\bibitem{LL}
{\sc R.~Lattes and J.-L. Lions}, {\em The Method of Quasi-Reversibility: Applications to Partial Differential Equations}, Elsevier, New York, 1969.

\bibitem{LRS}
{\sc M.~M. Lavrentiev, V.~G. Romanov, and S.~P. Shishatskii}, {\em Ill-Posed Problems of Mathematical Physics and Analysis}, AMS, Providence, R. I., 1986.

\bibitem{Le}
{\sc T.~T. Le, M.~V. Klibanov, L.~H. Nguyen, A.~Sullivan, and L.~Nguyen}, {\em Carleman contraction mapping for a 1d inverse scattering problem with experimental time-dependent data}, Inverse Problems, 38 (2022), p.~045002.

\bibitem{LN}
{\sc T.~T. Le and L.~H. Nguyen}, {\em A convergent numerical method to recover the initial condition of nonlinear parabolic equations from lateral cauchy data}, Journal of Ill-posed and Inverse Problems, 30 (2022), pp.~256--286.

\bibitem{Lee}
{\sc W.~Lee, S.~Liu, H.~Tembine, W.~Li, and S.~Osher}, {\em Controlling propagation of epidemics via mean-field control}, SIAM J. Appl. Math., 81 (2021), pp.~190--207.

\bibitem{Nguyen}
{\sc L.~H. Nguyen}, {\em The carleman contraction mapping method for quasilinear elliptic equations with over-determined boundary data}, Acta Math. Vietnam., 48 (2023), pp.~401--422.

\bibitem{KN}
{\sc L.~H. Nguyen and M.~V. Klibanov}, {\em Carleman estimates and the contraction principle for an inverse source problem for nonlinear hyperbolic equations}, Inverse Problems, 38 (2022), p.~035009.

\bibitem{Nov1}
{\sc R.~G. Novikov}, {\em The inverse scattering problem on a fixed energy level for the two-dimensional schr\"{o}dinger operator}, J. Funct. Anal., 103 (1992), pp.~409--463.

\bibitem{Nov2}
{\sc R.~G. Novikov}, {\em The $\bar{\partial}-$approach to approximate inverse scattering at fixed energy in three dimensions}, International Math. Research Peports, 6 (2005), pp.~287--349.

\bibitem{Rom}
{\sc V.~G. Romanov}, {\em Inverse Problems of Mathematical Physics}, VNU Press, Utrecht, The Netherlands, 1986.

\bibitem{T}
{\sc A.~N. Tikhonov, A.~V. Goncharsky, V.~V. Stepanov, and A.~G. Yagola}, {\em Numerical methods for the solution of Ill-posed problems}, Kluwer, London, 1995.

\end{thebibliography}

\end{document}